\documentclass{conm-p-l}
\usepackage{amsmath,amssymb,amsfonts,amsthm}
\usepackage[dvipsnames]{xcolor}
\usepackage{url}
\usepackage{array,enumerate}
\usepackage{booktabs}
\usepackage{tikz}
\usepackage{xypic}
\usepackage[backref]{hyperref}
\usepackage{color}
\definecolor{mylinkcolor}{rgb}{0.5,0.0,0.0}
\definecolor{myurlcolor}{rgb}{0.0,0.0,0.75}
\hypersetup{colorlinks=true,urlcolor=myurlcolor,citecolor=myurlcolor,linkcolor=mylinkcolor,linktoc=page,breaklinks=true}

\DeclareMathAlphabet{\mathpzc}{OT1}{pzc}{m}{it}

\newcommand{\Fp}{\mathbb{F}_p}

\newcommand{\Z}{\mathbb{Z}}
\newcommand{\Q}{\mathbb{Q}}
\newcommand{\R}{\mathbb{R}}
\newcommand{\C}{\mathbb{C}}

\newcommand{\Qbar}{\overline{\mathbb{Q}}}

\newcommand{\Gal}{{\rm Gal}}
\newcommand{\Jac}{\operatorname{Jac}}
\newcommand{\ST}{\operatorname{ST}}
\newcommand{\Trace}{\operatorname{Trace}}

\newcommand{\GSp}{{\rm GSp}}

\newcommand{\Unitary}{{\rm U}}
\newcommand{\USp}{{\rm USp}}

\newcommand{\SU}{{\rm SU}}
\newcommand{\Frob}{{\rm Frob}}
\newcommand{\End}{{\rm End}}
\newcommand{\tr}{\operatorname{tr}}
\DeclareMathOperator{\M}{\mathsf{M}}
\newcommand{\cyc}{\rm C}
\newcommand{\dih}{\rm D}
\newcommand{\p}{\mathfrak p}

\newtheorem{theorem}{Theorem}[section]
\newtheorem{conjecture}[theorem]{Conjecture}
\newtheorem{corollary}[theorem]{Corollary}
\newtheorem{lemma}[theorem]{Lemma}
\theoremstyle{definition}
\newtheorem{definition}[theorem]{Definition}

\newtheorem{remark}[theorem]{Remark}

\newtheorem{proposition}[theorem]{proposition}

\title{Sato-Tate groups of $y^2=x^8+c$ and $y^2=x^7-cx$.}

\begin{document}

\author{Francesc Fit\'e}
\address{Universit\"at Duisburg-Essen/Institut f\"ur Experimentelle Mathematik, Fakult\"at f\"ur Mathematik,
D-45127, Essen, Germany}
\curraddr{}
\email{francesc.fite@gmail.com}
\thanks{Fit\'e received support from the German Research Council via SFB 701 and SFB/TR 45.}

\author{Andrew V. Sutherland}
\address{Department of Mathematics, Massachusetts Institute of Technology, 77 Mass. Ave., Cambridge, MA  02139, USA}
\curraddr{}
\email{drew@math.mit.edu}
\urladdr{http://math.mit.edu/~drew} 
\thanks{Sutherland received support from NSF grant DMS-1115455.}

\begin{abstract}
We consider the distribution of normalized Frobenius traces for two families of genus 3 hyperelliptic curves over $\Q$ that have large automorphism groups: $y^2=x^8+c$ and $y^2=x^7-cx$ with $c\in\Q^*$.
We give efficient algorithms to compute the trace of Frobenius for curves in these families at primes of good reduction.
Using data generated by these algorithms, we obtain a heuristic description of the Sato-Tate groups that arise, both generically and for particular values of $c$.
We then prove that these heuristic descriptions are correct by explicitly computing the Sato-Tate groups via the correspondence between Sato-Tate groups and Galois endomorphism types.
\end{abstract}
\maketitle
\tableofcontents

\section{Introduction}

In this paper we consider two families of hyperelliptic curves over $\Q$:
\[
C_1\colon y^2=x^8+c,\qquad C_2\colon y^2=x^7-cx.
\]
For $c\in\Q^*$, these equations define hyperelliptic curves of genus 3 with good reduction at primes $p>3$ for which $v_p(c)=0$ (in fact,~$C_1$ also has good reduction at $3$).
For each such $p$ we have the \emph{trace of Frobenius}
\[
t_p(C_i) := p+1-\#\overline{C}_i(\Fp),
\]
where $\overline{C}_i$ denotes the reduction of $C_i$ modulo $p$.
From the Weil bounds, we know that $t_p$ lies in the interval $[-6\sqrt{p},6\sqrt{p}]$.
We wish to study the distribution of normalized Frobenius traces $t_p/\sqrt{p}\in [-6,6]$, as $p$ varies over primes of good reduction up to a bound $N$.

The generalized Sato-Tate conjecture predicts that as $N\to \infty$ this distribution converges to the distribution of traces in the \emph{Sato-Tate group}, a compact subgroup of $\USp(6)$ associated to the Jacobian of the curve.
For the two families considered here, the curves $C_i$ have Jacobians that are $\Q$-isogenous to the product of an elliptic curve and an abelian surface.\footnote{As we shall see, this abelian surface may itself be $\Q$-isogenous to a product of elliptic curves and is in any case never simple over $\Qbar$.}
This allows us to apply the classification of Sato-Tate groups for abelian surfaces obtained in \cite{FKRS12} to determine the Sato-Tate groups that arise.
This is achieved in \S\ref{section: det ST}.

After recalling the definition of the Sato-Tate group of an abelian variety in~\S\ref{section: background}, 
we begin in~\S\ref{section: traces} by deriving formulas for the Frobenius trace $t_p(C_i)$ in terms of the Hasse-Witt matrix of $\overline{C}_i$.
These formulas allow us to design particularly efficient algorithms for computing $t_p(C_i)$.
In \S\ref{section: ST guess}, under the assumption of the Sato-Tate conjecture, we use the numerical data obtained by applying these algorithms to heuristically guess the isomorphism class of the Sato-Tate groups of $C_1$ and $C_2$.
The explicit computation in \S\ref{section: det ST} proves that, in fact, these guesses are correct, without appealing to the Sato-Tate conjecture.

Strictly speaking, \S\ref{section: ST guess} and \S\ref{section: det ST} are independent of each other. However, we should emphasize that in the process of achieving our results, there was a constant and mutually beneficial interplay between the two distinct approaches. 

Up to dimension 3, the Sato-Tate group of an abelian variety defined over a number field $k$ is determined by its ring of endomorphisms over an algebraic closure of $k$.
Although the Sato-Tate group does not capture the ring structure of the endomorphisms, it does codify the $\R$-algebra generated by the endomorphism ring, and the structure of this $\R$-algebra as a Galois module, what we refer to as the \emph{Galois endomorphism type} of the abelian variety.
As an example, in \S\ref{section: GT} we compute the Galois endomorphism type of the Jacobian of $C_2$.

The problem of analysing the Frobenius trace distributions and determining the Sato-Tate groups that arise in these two families was originally posed as part of a course given by the authors at the winter school \emph{Frobenius Distributions on Curves} held in February, 2014, at the Centre International de Rencontres Math\'ematiques in Luminy.
This problem turned out to be more challenging than we anticipated (the analogous question in genus $2$ is quite straight-forward); this article represents a solution.

\subsection{Acknowledgements}
Both authors are grateful to the Centre International de Rencontres Math\'ematiques for the hospitality and financial support provided, and to the anonymous referee.

\section{Background}\label{section: background}

We start by briefly recalling the definition of the Sato-Tate group of an abelian variety $A$ defined over a number field $k$, and set some notation.
For a more detailed presentation we refer to \cite[Chap.\ 8]{Ser12} or \cite[\S2]{FKRS12}.

\subsection{The Sato-Tate group of an abelian variety}
Let $\overline k$ denote a fixed algebraic closure of $k$, and let $g$ be the dimension of $A$.
For each prime $\ell$ we have a continuous homomorphism
$$
\varrho_{A,\ell}\colon \Gal(\overline k/k) \rightarrow \GSp_{2g}(\Q_\ell)
$$
arising from the action of $\Gal(\overline k/k)$ on the rational Tate module $(\varprojlim A[\ell^n])\otimes\Q$.
Here $\GSp$ denotes the group of symplectic similitudes, which preserve a symplectic form up to a scalar; in our setting the preserved symplectic form arises from the Weil pairing.
Let $G_\ell$ be the Zariski closure of the image of $\varrho_{A,\ell}$, and let $G_\ell^1$ be the kernel of the similitude character $G_\ell\rightarrow \Q_\ell^*$.
We now choose an embedding $\iota\colon \Q_\ell \hookrightarrow \C$, and for each prime ideal $\p$ of the ring of integers of $k$, let $\Frob_\p$ denote an arithmetic Frobenius at~$\p$ and let $N(\p)$ be the cardinality of its residue field. 

\begin{definition}
The \emph{Sato-Tate group of $A$}, denoted $\ST(A)$, is a maximal compact subgroup of $G_\ell^1\otimes_\iota \C$.
For each prime $\p$ of good reduction for $A$, let $s(\p):=\varrho_{A,\ell}(\Frob_\p)\otimes_\iota N(\p)^{-1/2}$. 
\end{definition}

Let $\USp(2g)$ denote the group of $2g\times 2g$ complex matrices that are unitary and preserve a fixed symplectic form; this is a real Lie group of dimension $g(2g+1)$.
One can show that $\ST(A)$ is well-defined up to conjugacy in $\USp(2g)$, and that $s(\p)$ determines a conjugacy class in $\ST(A)$.

\begin{conjecture}[generalized Sato-Tate] Let $X$ denote the set of conjugacy classes of $\ST(A)$. Then:
\begin{enumerate}[(i)]
\item The conjugacy class of $\ST(A)$ in $\USp(2g)$ and the conjugacy classes $s(\p)$ in $\ST(A)$ are independent of the choice of the prime $\ell$ and the embedding $\iota$. 
\item When the primes $\p$ are ordered by norm, the $s(\p)$ are equidistributed on $X$ with respect to the projection of the Haar measure of $\ST(A)$ on $X$.
\end{enumerate}
\end{conjecture}

It follows from \cite{BK15} that part $(i)$ of the above conjecture is true for $g\leq 3$.
We next summarize some basic properties of the Sato-Tate group that we will need in our forthcoming discussion.
If $L/k$ is a field extension, we write~$A_L$ for the base change of~$A$ to~$L$. We denote by $K_A$ the minimal extension $L/k$ over which all the endomorphisms of $A$ are defined, that is, the minimal extension for which $\End(A_{L})\simeq \End(A_{\bar k})$.

The Sato-Tate group $\ST(A)$ is a compact real Lie group, but it need not be connected.
We use $\ST^0(A)$ to denote the connected component of the identity.

\begin{proposition}[Prop.\ 2.17 of \cite{FKRS12}]\label{proposition: 217} If $g\leq 3$, then the group of connected components $\ST(A)/\ST^0(A)$ is isomorphic to $\Gal(K_{A}/k)$.
\end{proposition}

\noindent
This proposition implies, in particular, that a prime $\p$ of good reduction for~$A$ splits completely in $K_A$ if and only if $s(\p)\in \ST^0(A)$.
One can in fact show a little bit more: for any algebraic extension $L/k$, the Sato-Tate group $\ST(A_L)$ is a subgroup of $\ST(A)$ with $\ST^0(A_L)=\ST^0(A)$ and
\[
\ST(A_L)/\ST^0(A_L)\simeq \Gal(K_A/(K_A\cap L))\subseteq \Gal(K_A/k).
\]

\subsection{Galois endomorphism types}
We now work in the category $\mathcal C$ of pairs $(G,E)$, where $G$ is a finite group and $E$ is an $\R$-algebra equipped with an $\R$-linear action of $G$.
A morphism $\Phi\colon (G,E)\rightarrow (G',E')$ of $\mathcal C$ consists of a pair $\Phi:=(\phi_1,\phi_2)$, where $\phi_1\colon G\rightarrow G'$ is a morphism of groups and $\phi_2\colon E\rightarrow E'$ is an equivariant morphism of $\R$-algebras, that is, 
$$
\phi_2(\phi_1(g)e)=\phi_2(g)(\phi_1(e)) \qquad\text{for all $g\in G$ and $e\in E$.}
$$ 

\begin{definition} The \emph{Galois endomorphism type} of $A$ is the isomorphism class in $\mathcal C$ of the pair $(\Gal(K_A/k),\End(A_{K_A})\otimes_\Z \R)$.
\end{definition}

By \cite[Prop. 2.19]{FKRS12}, for $g\leq 3$, the Galois endomorphism type is determined by the Sato-Tate group (in fact, the proof of this statement is effective, as we will illustrate in \S\ref{section: GT}). This result admits a converse statement at least for~$g\leq 2$.

\begin{theorem}[Thm.\ 4.3 of \cite{FKRS12}]\label{theorem: dictionary} For fixed $g\leq 2$, the Sato-Tate group and the Galois endomorphism type of an abelian variety $A$ defined over a number field~$k$ uniquely determine each other.
For $g=1$ (resp. $g=2$) there are $3$ (resp. $52$) possibilities for the Galois endomorphism type, all of which arise for some choice of $A$ and $k$.
\end{theorem}

For $g=1$ the $3$ possible Sato-Tate groups are $\SU(2)=\USp(2)$, a copy of the unitary group $\Unitary(1)$ embedded in $\SU(2)$, and its normalizer in $\SU(2)$; these arise, respectively, for elliptic curves $E/k$ without CM, with CM by a field contained in~$k$, and with CM by a field not contained in~$k$.
For $g=2$ a complete list of the $52$ possible Sato-Tate groups can be found in \cite{FKRS12}.

In order to simplify the notation, when $C$ is a smooth projective curve defined over the number field $k$, we may simply write 
$$
\ST(C):=\ST(\Jac(C)),\qquad\ST^0(C):=\ST^0(\Jac(C)),\qquad{\text{and}}\qquad K_{C}:=K_{\Jac(C)}\,.
$$

\section{Trace formulas}\label{section: traces}

Let $\overline{C}/\Fp$ be a smooth projective curve of genus $g\ge 1$ defined by an equation of the form $y^2=f(x)$ with $f\in \Fp[x]$ squarefree.
Let $n=(p-1)/2$ and let $f_k^n$ denote the coefficient of $x^k$ in the polynomial $f(x)^n$.
The \emph{Hasse--Witt} matrix of $\overline{C}$ is the $g\times g$ matrix $W_p:=[w_{ij}]$ over $\Fp$, where
\[
w_{ij}:=f^n_{ip-j}\qquad(1\le i,j\le g).
\]
It is shown in \cite{M61,Y78} that the characteristic polynomial $\chi(\lambda)$ of the Frobenius endomorphism of $\Jac(\overline C)$ satisfies
\[
\chi(\lambda) \equiv (-1)^g\lambda^g\det (W_p-\lambda I)\bmod p.
\]
In particular,
\[
\tr W_p\equiv t_p\bmod p,
\]
where $t_p:=p+1-\#\overline{C}(\Fp)$ is the trace of Frobenius.
The Weil bounds imply $|t_p|\le 2g\sqrt{p}$, which means that for all $p\ge 16g^2$, the trace of $W_p$ uniquely determines the integer $t_p$.

Let us now specialize to the case where $f(x)$ has the form
\[
f(x)=ax^d+bx^e,
\]
with $d\in \{2g+1,2g+2\}$, $e\in \{0,1\}$, and $a,b\in\Fp^*$; this includes the families $C_i$ defined in \S 1.
Writing
\[
f(x)^n = x^{en}(ax^{d-e}+b)^n
\]
and applying the binomial theorem yields
\[
f^n_{en+(d-e)r} = \binom{n}{r}a^rb^{n-r},
\]
and we have $f^n_k=0$ whenever $k$ is not of the form $k=en+(d-e)r$.
Setting $k=ip-j=i(2n+1)-j$ and solving for $r=r_{ij}$ yields
\[
r_{ij} := \frac{(2i-e)n+i-j}{d-e}\qquad (1\le i,j\le g).
\]
The entries of the Hasse-Witt matrix for $y^2=ax^d+bx^e$ are thus given by
\begin{equation}\label{eq:wij}
w_{ij} =
\begin{cases}
 \binom{n}{r_{ij}}a^{r_{ij}}b^{n-r_{ij}}& \text{if } r_{ij}\in\Z,\\
 0&\text{otherwise}.
\end{cases}
\end{equation}

For any fixed integer $i\in [1,g]$, the quantity $(2i-e)n+i-j$ lies in an interval of width $g-1<(d-e)/2$, as $j$ varies over integers in $[1,g]$.
This implies that at most one entry $w_{ij}$ in each row of $W_p$ is nonzero, and for this entry $r_{ij}$ is simply the nearest integer to $(2i-e)n/(d-e)$.

We now specialize to the two families of interest and assume $p>3$.
For $C_1\colon y^2=x^8+c$ we have $d=8, e=0, a=1$, and $b=\overline{c}$, where $\overline{c}$ denotes the image of $c$ in $\Fp$.  We thus have
\[
r_{ij}=\frac{2in+i-j}{8} = \frac{ip-j}{8}.
\]
For integers $i,j\in [1,3]$, the integral values of $r_{ij}$ that arise are listed below:
\begin{center}
\setlength\extrarowheight{2pt}
\begin{tabular}{llll}
$p\equiv 1\bmod 8:$&$r_{11}=\frac{n}{4},$&$r_{22}=\frac{n}{2},$&$r_{33}=\frac{3n}{4}$;\\
$p\equiv 3\bmod 8:$&$r_{13}=\frac{n-1}{4},$&$r_{31}=\frac{3n+1}{4}$;&\\
$p\equiv 5\bmod 8:$&$r_{22}=\frac{n}{2}$;&&\\
$p\equiv 7\bmod 8:$&none.&&\\
\end{tabular}
\end{center}
\smallskip

\noindent
This yields the following formulas for the trace of Frobenius:
\begin{equation}\label{eqtrC1}
t_p(C_1) \equiv_p \begin{cases}
\binom{n}{n/2}\overline{c}^{n/2}+\binom{n}{n/4}\overline{c}^{n/4}+\binom{n}{n/4}\overline{c}^{3n/4}&\text{if } p\equiv 1\bmod 8,\\
\binom{n}{n/2}\overline{c}^{n/2}&\text{if }p\equiv 5\bmod 8,\\
0&\text{otherwise}.\\
\end{cases}
\end{equation}

For $C_2\colon y^2=x^7-cx$ we have $d=7, e=1, a=1$, and $b=-\overline{c}$.
We thus have
\[
r_{ij}=\frac{(2i-1)n+i-j}{6}.
\]
For integers $i,j\in [1,3]$, the integral values of $r_{ij}$ that arise are listed below:
\begin{center}
\setlength\extrarowheight{2pt}
\begin{tabular}{llll}
$p\equiv 1\bmod 12:$&$r_{11}=\frac{n}{6},$&$r_{22}=\frac{n}{2},$&$r_{33}=\frac{5n}{6}$;\\
$p\equiv 5\bmod 12:$&$r_{13}=\frac{n-2}{6},$&$r_{22}=\frac{n}{2}$,&$r_{31}=\frac{5n+2}{6}$;\\
$p\equiv 7\bmod 12:$&none;&&\\
$p\equiv 11\bmod 12:$&none.&&\\
\end{tabular}
\end{center}
\smallskip

\noindent
This yields the following formulas for the trace of Frobenius:
\begin{equation}\label{eqtrC2}
t_p(C_2) \equiv_p \begin{cases}
\binom{n}{n/2}(-\overline{c})^{n/2}+\binom{n}{n/6}(-\overline{c})^{n/6}+\binom{n}{n/6}(-\overline{c})^{5n/6}&\text{if } p\equiv 1\bmod 12,\\
\binom{n}{n/2}(-\overline{c})^{n/2}&\text{if }p\equiv 5\bmod 12,\\
0&\text{otherwise}.\\
\end{cases}
\end{equation}

\subsection{Algorithms}\label{section: algorithms}

Computing the powers of $\overline{c}$ that appear in the formulas \eqref{eqtrC1} and \eqref{eqtrC2} for $t_p(C_i)$ is straight-forward; using binary exponentiation this requires just $O(\log p)$ multiplication in $\Fp$.
The only potential difficulty is the computation of the binomial coefficients $\binom{n}{n/2},\binom{n}{n/4},\binom{n}{n/6}$ modulo $p$, where $n=(p-1)/2$ and $p$ is known to lie in a suitable residue class.
Fortunately, there are very efficient formulas for computing these particular binomial coefficients modulo suitable primes $p$.
These are given by the lemmas below, in which $(\frac{2}{p})\in\{\pm 1\}$ denotes the Legendre symbol, and $m$, $x$, and $y$ denote integers.

\begin{lemma}
Let $p=4m+1=x^2+y^2$ be prime, with $x\equiv-\bigl(\frac{2}{p}\bigr)\bmod 4$.  Then
\[
\binom{2m}{m} \equiv 2(-1)^{m+1}x\bmod p.
\]
\end{lemma}
\begin{proof}
See \cite[Thm.\ 9.2.2]{BEW98}.
\end{proof}

\begin{lemma}
Let $p=8m+1=x^2+2y^2$ be prime, with $x\equiv-\bigl(\frac{2}{p}\bigr)\bmod 4$.  Then
\[
\binom{4m}{m} \equiv 2(-1)^{m+1}x\bmod p.
\]
\end{lemma}
\begin{proof}
See \cite[Thm.\ 9.2.8]{BEW98}.
\end{proof}

\begin{lemma}
Let $p=12m+1=x^2+y^2$ be prime, with $x\equiv-\bigl(\frac{2}{p}\bigr)\bmod 4$, and define $\epsilon$ to be $0$ if $x\equiv 0\bmod 3$ and $1$ otherwise.  Then
\[
\binom{6m}{m} \equiv 2 (-1)^{m+\epsilon}x\bmod p.
\]
\end{lemma}
\begin{proof}
See \cite[Thm.\ 9.2.10]{BEW98} (replace $\rho_4^2$ with $(-1)^{\epsilon-1}$).
\end{proof}

To apply these lemmas, one uses Cornacchia's algorithm to find a solution $(x,y)$ to $p=x^2+dy^2$, where $d=1$ when computing $\binom{n}{n/2}\bmod p$ or $\binom{n}{n/6}\bmod p$, and $d=2$ when computing $\binom{n}{n/4}$.
Cornacchia's algorithm requires as input a square-root $\delta$ of $-d$ modulo $p$ (if no such $\delta$ exists then $p=x^2+dy^2$ has no solutions).
\smallskip

\noindent
\textsc{Cornacchia's Algorithm}\\
\noindent Given integers $1\le d < m$ and an integer $\delta\in[0,m/2]$ such that $\delta^2\equiv -d\bmod m$, find a solution $(x,y)$ to $x^2+dy^2=m$ or determine that none exist as follows:
\smallskip

\begin{enumerate}[1.]
\item Set $x_0:=m$, $x_1:=\delta$, and $i=1$.
\item While $x_i^2 \ge m$, set $x_{i+1}:=x_{i-1}\bmod x_i$ with $x_{i+1}\in[0,x_i)$ and increment $i$.
\item If $(m-x_i^2)/d=y^2$ for some $y\in \Z$, output the solution $(x_i,y)$.\\
Otherwise, report that no solution exists.
\end{enumerate}
\smallskip

\noindent
See \cite{Bas04} for a simple proof of the correctness of this algorithm.
We now consider its computational complexity, using $\M(n)$ to denote the time to multiply two $n$-bit integers; we may take $\M(n)=O(n\log n\log\log n)$ via \cite{SS71}.
The first two steps correspond to half of the standard Euclidean algorithm for computing the GCD of $m$ and $\delta$, whose bit-complexity is bounded by $O(\log^2 m)$; see \cite[Thm.\ 3.13]{GG13}.
The time required in step 3 to perform a division and check whether the result is a square integer is also $O(\M(\log m))$; see \cite[Thm.\ 9.8, Thm.\ 9.28]{GG13}).
Thus the overall complexity is $O(\log^2 m)$, the same as the Euclidean algorithm.

\begin{remark}
There is an asymptotically faster version of the Euclidean algorithm that allows one to compute any particular pair of remainders $(x_{i-1},x_i)$, including the unique pair for which $x_{i-1}\ge \sqrt{m}>x_i$, in $O(\M(\log m)\log\log m)$ time; see \cite{PW03}.
This yields a faster version of Cornacchia's algorithm that runs in quasi-linear time, but we will not use this.
\end{remark}

We now turn to the problem of computing the square-root $\delta$ of $-d\bmod m$ that is required by Cornacchia's algorithm.
There are two basic strategies for doing this:

\begin{enumerate}[1.]
\setlength{\itemsep}{8pt}
\item (Cipolla-Lehmer) Use a probabilistic root-finding algorithm to factor $x^2+d$ in $\Fp[x]$.  This takes
$O(\M(\log p)\log p)$ expected time.

\item (Tonelli-Shanks) Given a generator $g$ for the 2-Sylow subgroup of $\Fp^*$, compute the discrete logarithm $e$ of $(-d)^s\in \langle g\rangle$ and let $\delta = g^{-e/2}(-d)^{(s+1)/2},$ where $p=2^vs+1$ with $s$ odd.
This takes $O(\M(\log p)(\log p + v\log v/\log\log v))$ time if the algorithm in \cite{Sut11} is used to compute the discrete logarithm.
\end{enumerate}

We will exploit both approaches.  To obtain a generator for the 2-Sylow subgroup of $\Fp^*$ one may take $\alpha^s$ for any quadratic non-residue $\alpha$.
Half the elements of $\Fp^*$ are non-residues, so randomly selecting elements and computing Legendre symbols will yield a non-residue after 2 attempts, on average, and each attempt takes $O(\M(\log p)\log\log p)$ time, via \cite{BZ10}.
Unfortunately, we know of no efficient way to deterministically obtain a quadratic non-residue modulo $p$ without assuming the generalized Riemann hypothesis (GRH).
Under the GRH the least non-residue is $O((\log p)^2)$ \cite{Bach90}, thus if we simply test increasing integers $2,3,\ldots$  we can obtain a non-residue $\alpha$ for a total cost of $O(\M(\log p)\log^2\hspace{-1.5pt} p\log\log p)$.

But we are actually interested in computing $t_p(C_i)$ for many primes $p\le N$, for some large bound $N$; on average, this approach will find a non-residue very quickly.
As $N\to\infty$ the average value of the least non-residue converges to
\[
\sum_{k=1}^\infty \frac{p_k}{2^k}= 3.674643966\ldots,
\]
where $p_k$ denotes the $k$th prime, as shown by Erd\"os \cite{Erd61}.

Finally, we should mention an alternative approach to solving $p=x^2+dy^2$ that is completely deterministic.
Construct an elliptic curve $E/\Fp$ with complex multiplication by the imaginary quadratic order $\mathcal{O}$ with discriminant $D=-d$ (or $D=-4d$ if $-d\not\equiv 0,1\bmod 4$) and then use Schoof's algorithm \cite{Sch85} to compute the trace of Frobenius $t$ of $E$.
We then have $4p=t^2-v^2D$, since the Frobenius endomorphism with trace $t$ and norm $p$ corresponds to $\frac{t\pm v\sqrt{D}}{2}\in\mathcal{O}$, and therefore $(t/v)^2\equiv D\bmod p$.
If $D=-d$, we have a square root of $-d$ modulo $p$ and can use Cornacchia's algorithm to solve $p=x^2+dy^2$.
If $D=-4d$, then $t$ is even and $(t/2,2v)$ is already a solution to $p=x^2+dy^2$.
We are specifically interested in the cases $D=-4$ and $D=-8$.
For $D=-4$ we can take $E\colon y^2=x^3-x$, and for $D=-8$ we can take $E\colon y^2=x^4+4x^2+2x$; see \S 3 of Appendix A in \cite{Si94}.

We collect all of these observations in the following theorem.

\begin{theorem}\label{thm:bounds}
Let $C_1\colon y^2=x^8+c$ and $C_2\colon y^2=x^7-cx$ be as above.
Let $p>3$ be a prime with $v_p(c)=0$.
We can compute $t_p(C_i)$:
\begin{itemize}
\item probabilistically in $O(\M(\log p)\log p)$ expected time;
\item deterministically in $O(\M(\log p)\log^2\hspace{-1.5pt} p\log \log p)$ time, assuming GRH;
\item deterministically in $O(\M(\log^3\hspace{-1.5pt} p)\log^2\hspace{-1.5pt} p/\log \log p)$ time.
\end{itemize}
For any positive integer $N$, we can compute $t_p(C_i)$ for all $3<p\leq N$ with $v_p(c)=0$
deterministically in $O(N\M(\log N))$ time.
\end{theorem}
\begin{proof}
Since we are computing asymptotic bounds, we may assume $p\ge 144$ (if not, just count points na\"ively).
Then $t_p(C_i)\bmod p$ uniquely determines $t_p(C_i)\in\Z$.

For the first bound we use the Cipolla-Lehmer approach to probabilistically compute the square root required by Cornacchia's algorithm in $O(\M(\log p)\log p)$ time, matching the time required to apply any of Lemmas~2.1-4, and the time required by the exponentiations of $\overline{c}$ needed to compute $t_p(C_i)$.

For the second bound we instead use the Tonelli-Shanks approach to computing square roots, relying on iteratively testing increasing integers to find a non-residue.  Under the GRH this takes $O(\M(\log p)\log^2\hspace{-1.5pt} p\log\log p)$ time, which dominates everything else.

For the third bound, we instead use Schoof's approach to solve $p=x^2+dy^2$.
The analysis in \cite[Cor.\ 11]{SS14} shows that Schoof's algorithm can be implemented to run in $O(\M(\log^3\hspace{-1.5pt} p)\log^2\hspace{-1.5pt} p/\log\log p)$ time.

For the final bound we proceed as in the GRH bound but instead rely on the Erd\"os bound for the least non-residue modulo $p\le N$, on average.
By the prime number theorem there are $O(N/\log N)$ primes $p\le N$; the total number of quadratic residue tests is thus $O(N/\log N)$.  It takes $O(\M(\log N)\log\log N)$ time for each test, so the total time spent finding non-residues is $O(N\M(\log N)\log\log N/\log N)$.
The average $2$-adic valuation of $p-1$ over primes $p\le N$ is $O(1)$, so the total time spent computing square roots modulo primes $p\le N$ using the Tonelli-Shanks approach is $O((N/\log N)\M(\log N)\log N)) = O(N\M(\log N))$ which dominates the time spent finding non-residues and matches the time spent on everything else.
\end{proof}

We note that the average time per prime $p\le N$ using a deterministic algorithm is $O(\M(\log p)\log p)$, which matches the expected time when applying our probabilistic approach for any particular prime $p\le N$; both bounds are quasi-quadratic $O((\log p)^{2+o(1)})$.
For comparison, the average time per prime $p\le N$ achieved using the average polynomial time algorithm in \cite{HS1,HS2} is $O((\log p)^{4+o(1)})$.

\begin{remark}
Although Theorem~\ref{thm:bounds} only addresses the computation of $t_p(C_i)$, for $p\not\equiv 3\bmod 8$ (resp. $p\not\equiv 5\bmod 12$) we can readily compute the entire Hasse-Witt matrix $W_p$ for $C_1$ (resp. $C_2$) using the same approach and within the same complexity bounds.
\end{remark}

\section{Guessing Sato-Tate groups}\label{section: ST guess} 

In this section we analyze the Sato-Tate distributions of the curves $C_i$ and arrive at a heuristic characterization of their Sato-Tate groups up to isomorphism, based on statistics collected using the algorithm described in \S\ref{section: algorithms}.
In \S\ref{section: det ST} we will unconditionally prove that our heuristic characterizations are correct. 

\subsection{The Sato-Tate distribution of $C_1$}\label{section: heuC1}

Before applying any heuristics we can derive some information about the structure of the Sato-Tate group directly from the formulas developed in the previous section.
The possible shapes of the Hasse-Witt matrix for $C_1$ at a primes $p\equiv1,3,5,7\bmod 8$ are depicted below, with the residue class of $p\bmod 8$ in parentheses:
\smallskip

\[
\begin{bmatrix}*&0&0\\0&*&0\\0&0&*\end{bmatrix}(1),\quad \begin{bmatrix}0&0&*\\0&0&0\\ *&0&0\end{bmatrix}(3),\quad \begin{bmatrix}0&0&0\\0&*&0\\0&0&0\end{bmatrix}(5),\quad \begin{bmatrix}0&0&0\\0&0&0\\0&0&0\end{bmatrix}(7).
\]
\smallskip

\noindent
From this we can (unconditionally) conclude the following:
\begin{enumerate}
\item[(a)] the component group $\ST(C_1)/\ST^0(C_1)$ has order divisible by $4$;
\item[(b)] we have $s(p)$ in $\ST^0(C_1)$ only if $p\equiv 1\bmod 8$;
\item[(c)] the field $K_{C_1}$ contains $\Q(i,\sqrt{2})$.
\end{enumerate}
We note that (c) follows immediately from (b): a prime $p>2$ splits completely in $\Q(i,\sqrt{2})$ if and only if $p\equiv 1\bmod 8$.

Table~\ref{table:C1moments} lists moment statistics $M_n$ for the curve $C_1\colon y^2=x^8+c$ for selected values of $c$, where $M_n$ is the average value of the $n$th power of the normalized $L$-polynomial coefficient
\[
a_1:=-t_p/\sqrt{p},
\]
over odd primes $p\le 2^{40}$ not dividing $c$.
The moment statistics $M_n$ for odd $n$ are all close to zero, so we list $M_n$ only for even $n$.

\begin{table}[bth!]
\setlength{\tabcolsep}{10pt}
\begin{tabular}{lrrrrr}
$c$ & $M_2$ & $M_4$ & $M_6$ & $M_8$ & $M_{10}$\\\midrule
$1$    & 3.000 & 50.999 & 1229.971 & 33634.058 & 978107.050\\
$2$    & 2.000 & 27.000 &  619.987 & 16834.560 & 489116.939\\
$3$    & 2.000 & 24.000 &  469.984 & 11234.520 & 297593.517\\
$4$    & 3.000 & 51.000 & 1229.990 & 33634.650 & 978125.742\\
$5$    & 2.000 & 23.999 &  469.976 & 11234.211 & 297585.653\\
$6$    & 2.000 & 23.999 &  469.979 & 11234.275 & 297587.173\\
$7$    & 2.000 & 23.999 &  469.968 & 11234.007 & 297579.866\\
$8$    & 2.000 & 27.000 &  619.987 & 16834.560 & 498116.939\\
$9$    & 2.000 & 27.000 &  619.991 & 16834.654 & 498118.664\\
$2^4$   & 3.000 & 50.999 & 1229.971 & 33634.058 & 978107.050\\
$3^3$   & 2.000 & 24.000 &  469.987 & 11234.520 & 297594.971\\
$2^5$   & 2.000 & 27.000 &  619.987 & 16834.560 & 498116.939\\
$2^6$   & 3.000 & 51.000 & 1229.990 & 33634.650 & 978125.742\\
$3^4$   & 3.000 & 51.000 & 1229.990 & 33634.593 & 978121.494\\
\midrule
\end{tabular}
\bigskip

\caption{Trace moment statistics for $C_1\colon y^2=x^8+c$ for $p\le 2^{40}$.}\label{table:C1moments}
\end{table}
\vspace{-8pt}

There appear to be three distinct trace distributions that arise, depending on whether the integer $c$ is in 
$$
\Q(i, \sqrt 2)^{*4}, \quad \Q(i, \sqrt 2)^{*2}\setminus \Q(i, \sqrt 2)^{*4}, \quad \text{or} \quad  \Q(i, \sqrt 2)^{*}\setminus \Q(i, \sqrt 2)^{*2};
$$
these can be distinguished by whether the nearest integer to $M_4$ is 51, 27, or 24, respectively.
Histogram plots of representative examples are shown with $c=1,2,3$ in Figure~\ref{figure:C1histograms}.
We note that in each histogram the central spike at $0$ has area 1/2, while the spikes at $-2$ and $2$ have area zero.

\begin{figure}[bth!]
\begin{center}
\includegraphics[scale=0.135]{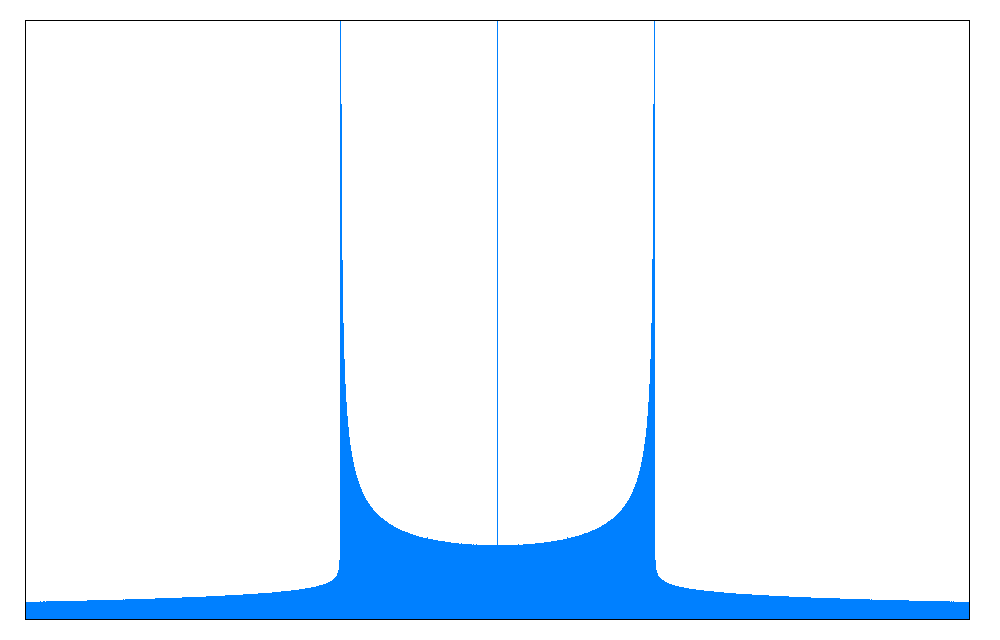}\hspace{22pt}
\includegraphics[scale=0.135]{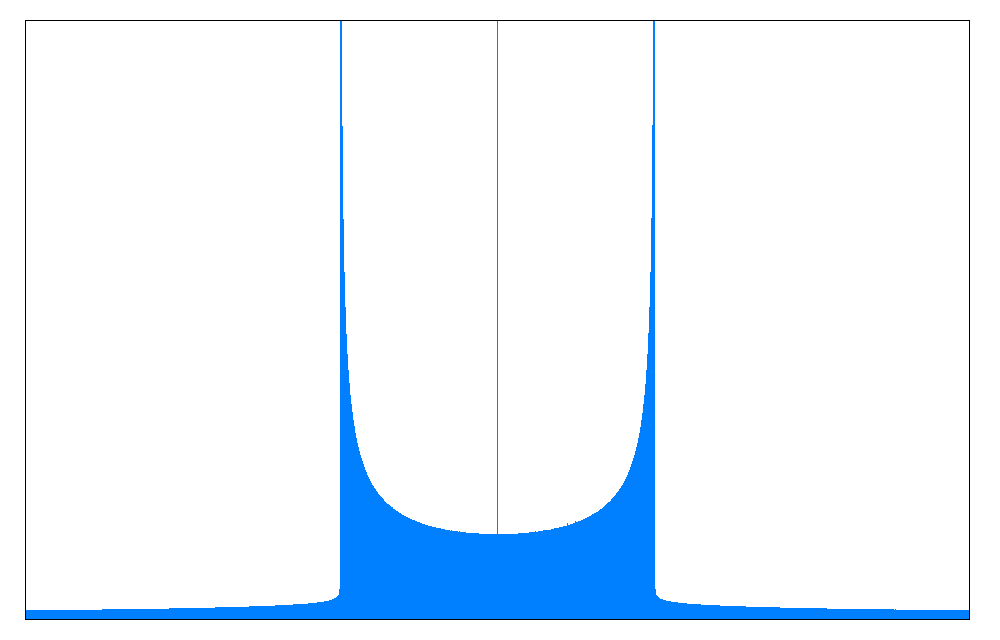}\hspace{22pt}
\includegraphics[scale=0.135]{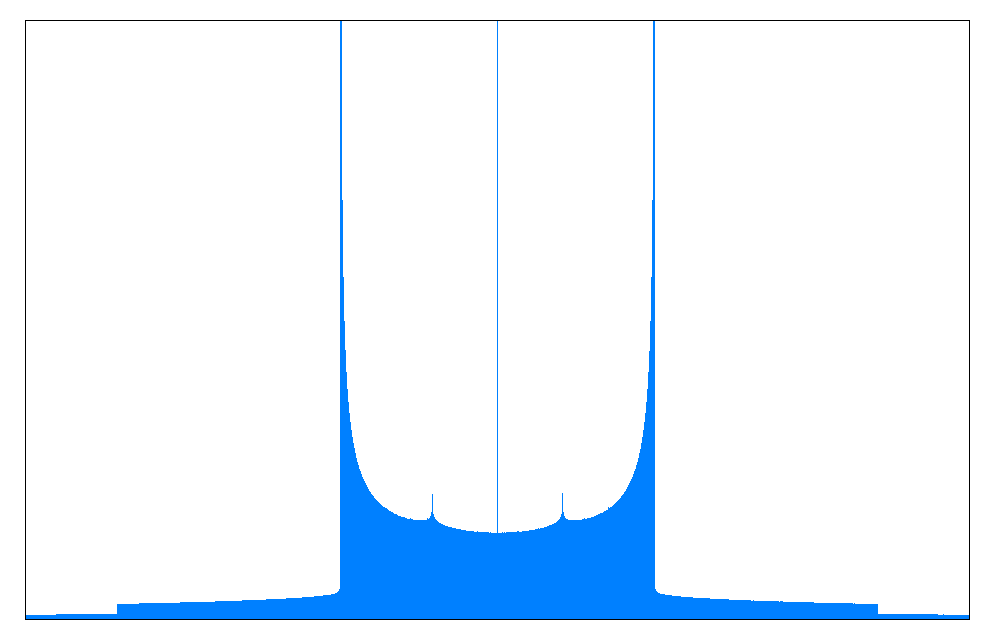}\\
$y^2=x^8+1$\hspace{77pt}$y^2=x^8+2$\hspace{77pt}$y^2=x^8+3$\\
\caption{$a_1$-histograms for three representative curves $C_1$.}\label{figure:C1histograms}
\end{center}
\end{figure}

Based on the data in Table~\ref{table:C1moments}, we expect $K_{C_1}$ to contain $\Q(i,\sqrt{2},\sqrt[4]{c})$.
If we now require $c$ to be a fourth-power and restrict to primes $p\equiv 1\bmod 8$, we can investigate the Sato-Tate distribution of $C_1$ over the number field $\Q(i,\sqrt{2}, \sqrt[4]{c})$.
For $c=1$ we obtain the moments listed below:
\medskip

\begin{center}
\setlength{\tabcolsep}{10pt}
\begin{tabular}{lrrrrr}
$c$ & $M_2$ & $M_4$ & $M_6$ & $M_8$ & $M_{10}$\\\midrule
$1$    & 10.000 & 197.997 & 4899.892 & 134466.452 & 3912182.569\\
\midrule
\end{tabular}
\end{center}
\medskip

\noindent
The corresponding histogram is shown in Figure~\ref{figure:C1identityhistogram}.

\begin{figure}[bth!]
\begin{center}
\includegraphics[scale=0.4]{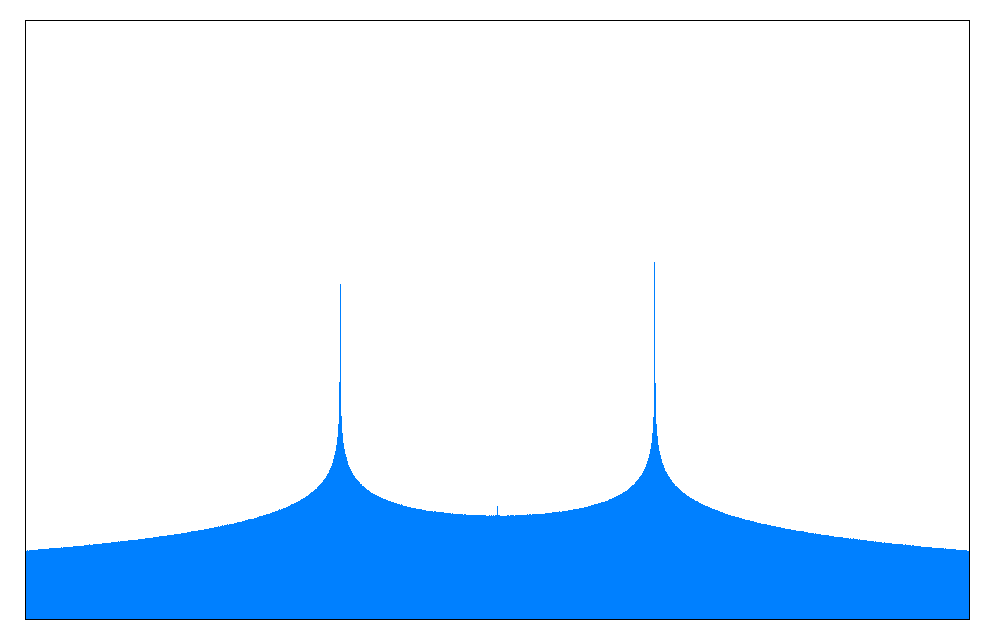}\vspace{-12pt}
\caption{$a_1$-histogram for $y^2=x^8+1$ over $\Q(i,\sqrt{2})$.}\label{figure:C1identityhistogram}
\end{center}
\end{figure}

We claim that this distribution corresponds to a connected Sato-Tate group, namely, the group
\[
\Unitary(1)_2\times \Unitary(1) := \left\langle\begin{bmatrix}U(u)&0&0\\0&U(u)&0\\0&0&U(v)\end{bmatrix}:u,v\in\Unitary(1)\right\rangle,
\]
where for $u\in\Unitary(1):=\{e^{i\theta}:\theta\in[0,2\pi)\}$ the matrix $U(u)$ is defined by
\begin{equation}\label{equation: defuni}
U(u):=\begin{bmatrix}u&0\\0&\overline{u}\end{bmatrix}.
\end{equation}
The $a_1$-moment sequence for $\Unitary(1)_2\times\Unitary(1)$ can be computed as the binomial convolution of the $a_1$-moment sequences for $\Unitary(1)_2$ and $\Unitary(1)$ given in \cite{FKRS12}.
Explicitly, if $M_n(G)$ denotes the $n$th moment of $a_1$ (or any class function), for $G= H_1\times H_2$, we have
\begin{equation}\label{eq:binconv}
M_n(G)=\sum_{k=0}^n\binom{n}{k}M_k(H_1)M_{n-k}(H_2).
\end{equation}
Applying this to $G=\Unitary(1)_2\times\Unitary(1)$ yields:
\medskip

\small
\begin{center}
\begin{tabular}{lrrrrrrrrrrrr}
\setlength{\tabcolsep}{0pt}
&{\footnotesize $M_0$}&{\footnotesize $M_1$}&{\footnotesize $M_2$}&{\footnotesize $M_3$}&{\footnotesize $M_4$}&{\footnotesize $M_5$}&{\footnotesize $M_6$}&{\footnotesize $M_7$}&{\footnotesize $M_8$}&{\footnotesize $M_9$}&{\small $M_{10}$}\\\midrule
$\Unitary(1)_2$&1&0&8&0&96&0&1280&0&17920&0&258048\\
$\Unitary(1)$&1&0&2&0&6&0&20&0&70&0&252\\
$\Unitary(1)_2\times\Unitary(1)$&1&0&10&0&198&0&4900&0&1344700&0&3912300\\
\midrule
\end{tabular}
\end{center}
\normalsize
\medskip

\noindent
This is in close agreement (within $0.1\%$) with the moment statistics for $y^2=x^8+1$ over $\Q(i,\sqrt{2})$.
We thus conjecture that the identity component is
\[
\ST^0(C_1) = \Unitary(1)_2\times\Unitary(1),
\]
up to conjugacy in $\USp(6)$, and
\[
K_{C_1}=\Q(i,\sqrt{2},\sqrt[4]{c}).
\]
For generic $c$ the component group of $\ST(C_1)$ is then isomorphic to
\[
\Gal(K_{C_1}/\Q)\simeq \dih_4\times\cyc_2,
\]
where $\dih_4$ is the dihedral group of order $8$ and $\cyc_2$ is the cyclic group of order $2$.

\subsection{The Sato-Tate distribution of $C_2$}\label{section: heuC2}

The possible shapes of the Hasse-Witt matrix for $C_2$ at a primes $p\equiv1,5,7,11\bmod 12$ are depicted below, with the residue class of $p\bmod 12$ in parentheses:
\smallskip

\[
\begin{bmatrix}*&0&0\\0&*&0\\0&0&*\end{bmatrix}(1),\quad \begin{bmatrix}0&0&*\\0&*&0\\{*}&0&0\end{bmatrix}(5),\quad \begin{bmatrix}0&0&0\\0&0&0\\0&0&0\end{bmatrix}(7),\quad \begin{bmatrix}0&0&0\\0&0&0\\0&0&0\end{bmatrix}(11).
\]
\smallskip

\noindent
From this information we can conclude that:
\begin{enumerate}
\item[(a)] the order of the component group $\ST(C_2)/\ST^0(C_2)$ is a multiple of $4$;
\item[(b)] we have $s(p)\in\ST^0(C_2)$ only if $p\equiv 1\bmod 12$;
\item[(c)] the field $K_{C_2}$ contains $\Q(i,\sqrt{3})$.
\end{enumerate}
We note that (c) follows immediately from (b): a prime $p>3$ splits completely in $\Q(i,\sqrt{3})$ if and only if $p\equiv 1\bmod 12$.

\begin{table}[bth!]
\setlength{\tabcolsep}{10pt}
\begin{tabular}{lrrrrr}
$c$ & $M_2$ & $M_4$ & $M_6$ & $M_8$ & $M_{10}$\\\midrule
1 & 3.000 & 62.999 & 1829.927 & 57434.041 & 1860104.868\\
2 & 2.000 & 29.999 &  719.982 & 20649.366 &  641569.043\\
3 & 2.000 & 29.999 &  719.972 & 20649.083 &  641561.180\\
4 & 2.000 & 30.000 &  719.985 & 20649.447 &  641572.217\\
5 & 2.000 & 30.000 &  719.988 & 20649.586 &  641578.161\\
6 & 2.000 & 30.000 &  720.004 & 20650.090 &  641593.419\\
7 & 2.000 & 30.000 &  719.991 & 20649.656 &  641579.324\\
8 & 3.000 & 62.999 & 1829.978 & 57434.221 & 1860110.123\\
9 & 2.000 & 29.999 &  719.973 & 20649.084 &  641561.181\\
$2^4$ & 2.000 & 30.000 &  719.985 & 20649.447 &  641572.217\\
$3^3$ & 3.000 & 62.999 & 1829.972 & 57434.041 & 1860104.867\\
$2^5$ & 2.000 & 29.999 &  719.982 & 20649.366 &  641569.043\\
$2^6$ & 3.000 & 62.999 & 1829.972 & 57434.041 & 1860104.868\\
$3^4$ & 2.000 & 29.999 &  719.973 & 20649.084 &  641561.181\\
\midrule
\end{tabular}
\bigskip

\caption{Trace moment statistics for $C_2\colon y^2=x^7-cx$ for $p\le 2^{40}$.}\label{table:C2moments}
\end{table}

Table~\ref{table:C2moments} lists moment statistics $M_n$ for the curve $C_2\colon y^2=x^7-cx$ for various values of $c$.
There now appear to be just two distinct trace distributions that arise, depending on whether the integer $c$ is a cube or not; these can be distinguished by whether the nearest integer to $M_2$ is 2 or 3, respectively.
Histogram plots of three representative examples are shown for $c=1,2$ in Figure~\ref{figure:C2histograms}.
In the histogram for $c=1$ the central spike at $0$ has area 1/2 and the spikes at $-2$ and $2$ have area zero, but in the histogram for $c=2$ the central spike has area 7/12, while the spikes at $-4,-2,2,4$ have area zero.
This gives us a further piece of information: the order of the component group $\ST(C_2)/\ST^0(C_2)$ should be divisible by 12.

\begin{figure}[bth!]
\begin{center}
\includegraphics[scale=0.135]{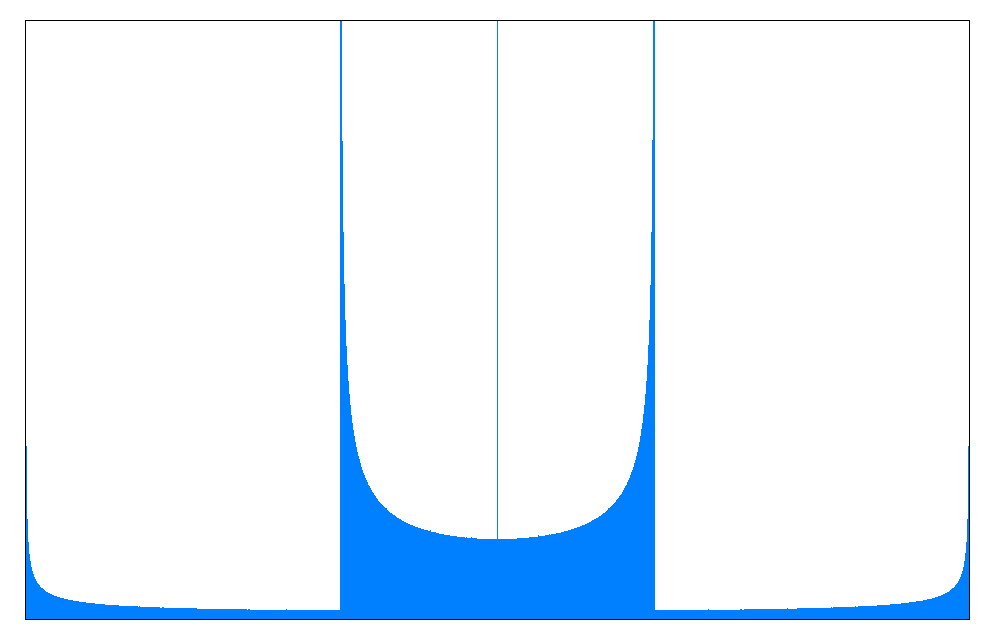}\hspace{80pt}
\includegraphics[scale=0.135]{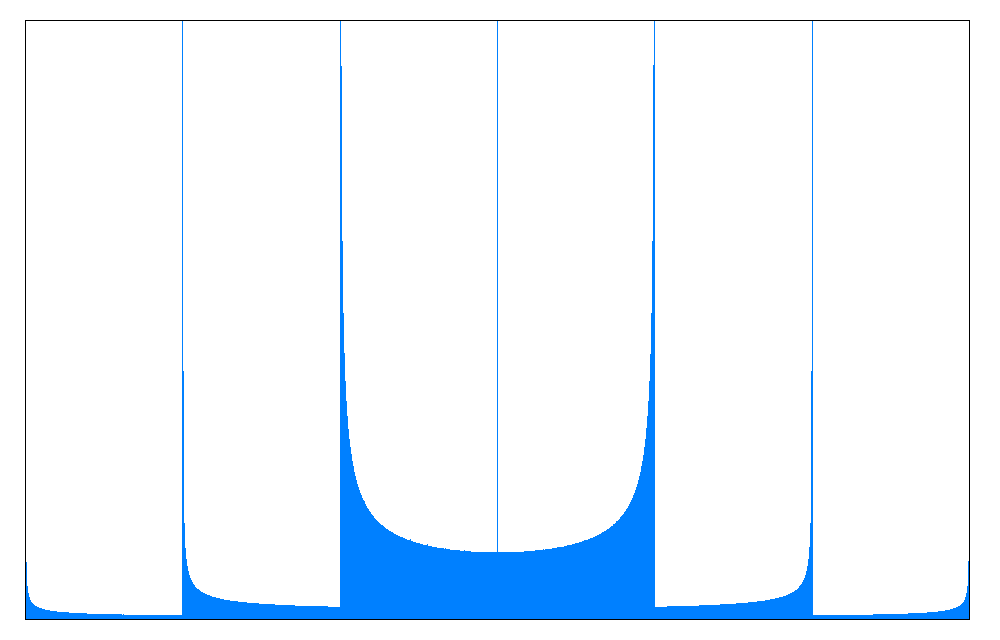}\\
$y^2=x^7-x$\hspace{125pt}$y^2=x^7-2x$\\
\caption{$a_1$-histograms for two representative curves $C_2$.}\label{figure:C2histograms}
\end{center}
\end{figure}

Based on the data in Table~\ref{table:C1moments}, we expect $K_{C_2}$ to contain $\Q(i,\sqrt{3},\sqrt[3]{c})$.
We now require $c$ to be a cube and restrict to primes $p\equiv 1\bmod 12$ in order to investigate the Sato-Tate distribution of $C_2$ over the number field $\Q(i,\sqrt{3}, \sqrt[3]{c})$.
For $c=1$ we obtain the moments listed below:
\medskip

\begin{center}
\setlength{\tabcolsep}{10pt}
\begin{tabular}{lrrrrr}
$c$ & $M_2$ & $M_4$ & $M_6$ & $M_8$ & $M_{10}$\\\midrule
$1$    & 10.000 & 245.997 & 7299.909 & 229666.846 & 7440189.620\\
\midrule
\end{tabular}
\end{center}
\medskip

\noindent
The corresponding histogram is shown in Figure~\ref{figure:C2notidentityhistogram}, and is clearly \emph{not} the distribution of the identity component; one can see directly that there are (at least) two components.

\begin{figure}[bth!]
\begin{center}
\includegraphics[scale=0.4]{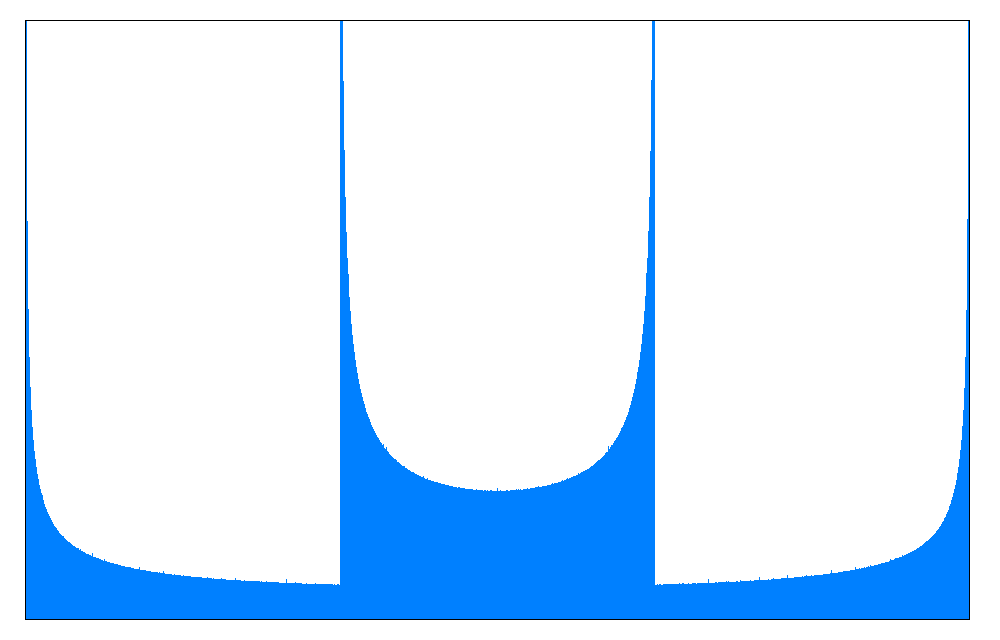}\vspace{-12pt}
\caption{$a_1$-histogram for $y^2=x^7-x$ over $\Q(i,\sqrt{3})$.}\label{figure:C2notidentityhistogram}
\end{center}
\end{figure}

\noindent
This suggests that we should try computing the Sato-Tate distribution over a quadratic extension of $\Q(i,\sqrt{3})$.
After a bit of experimentation, one finds that $\Q(i,\sqrt[4]{-3})$ works.  With $c=1$ we obtain the moments statistics:
\medskip

\begin{center}
\setlength{\tabcolsep}{10pt}
\begin{tabular}{lrrrrr}
$c$ & $M_2$ & $M_4$ & $M_6$ & $M_8$ & $M_{10}$\\\midrule
$1$    & 18.000 & 485.994 & 14579.770 & 459261.673 & 14880044.545\\
\midrule
\end{tabular}
\end{center}
\medskip

\noindent
The corresponding histogram is shown in Figure~\ref{figure:C2identityhistogram}.

\begin{figure}[bth!]
\begin{center}
\includegraphics[scale=0.4]{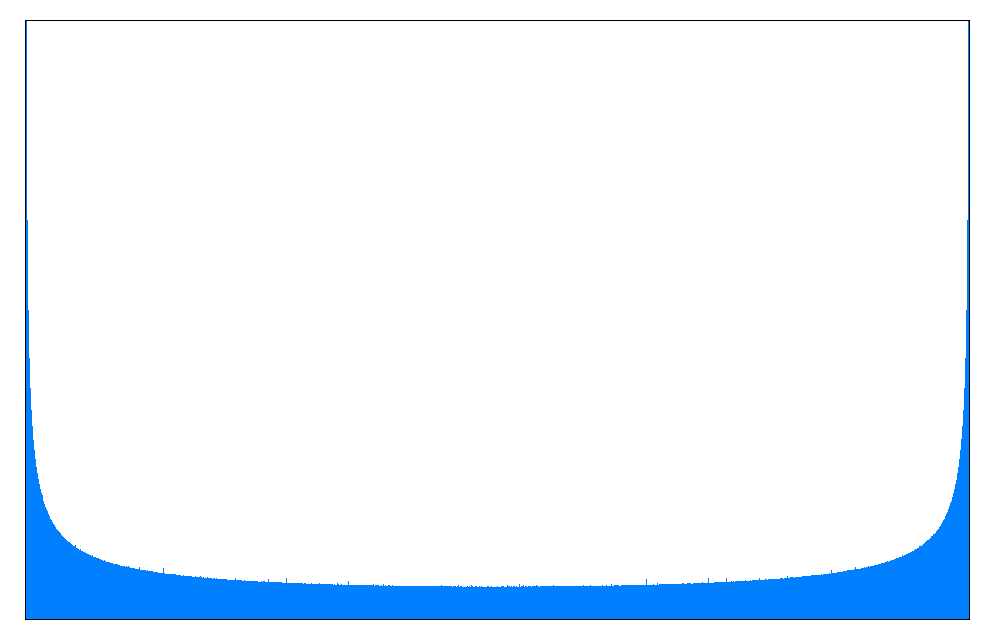}\vspace{-12pt}
\caption{$a_1$-histogram for $y^2=x^7-x$ over $\Q(i,\sqrt[4]{-3})$.}\label{figure:C2identityhistogram}
\end{center}
\end{figure}

We claim that this distribution corresponds to a connected Sato-Tate group, namely, the group
\[
\Unitary(1)_3 := \left\langle\begin{bmatrix}U(u)&0&0\\0&U(u)&0\\0&0&U(u)\end{bmatrix}:u\in\Unitary(1)\right\rangle.
\]
The $a_1$-moment sequence for $\Unitary(1)_3$ can be computed as the $3a_1$-moment sequence for $\Unitary(1)$, which simply scales the $n$th moment by $3^n$.  This yields the moments:
\medskip

\begin{center}
\begin{tabular}{lrrrrrr}
&$M_2$&$M_4$&$M_6$&$M_8$&$M_{10}$\\\midrule
$\Unitary(1)_3$&18&486&14580&459270&14880348\\
\midrule
\end{tabular}
\end{center}

\noindent
which are in close agreement (better than 0.1\%) with the moment statistics for $y^2=x^7-x$ over the field $\Q(i,\sqrt[4]{-3})$.

A complication arises if we repeat the experiment using a cube $c\ne 1$; we no longer get a connected Sato-Tate group!
Taking $c$ to be a sixth-power works, but we now need to ask whether, generically, the degree 48 extension $\Q(i,\sqrt[4]{-3},\sqrt[6]{c})$ is the minimal extension required to get a connected Sato-Tate group.
We have good reason to believe that a degree 24 extension \emph{is} necessary, since $K_{C_2}$ appears to properly contain the degree 12 field $\Q(i,\sqrt{3},\sqrt[3]{c})$, but it is not clear that a degree 48 extension is required.
We thus check various quadratic subextensions of $\Q(i,\sqrt[4]{-3},\sqrt[6]{c})$ and find that $\Q(i,\sqrt[3]{c},\sqrt{c\sqrt{-3}})$ works consistently.

We thus conjecture that
\[
K_{C_2}=\Q(i,\sqrt[3]{c},\sqrt{c\sqrt{-3}}).
\]
This implies that for generic $c$, the component group of $\ST(C_2)$ is isomorphic to
\[
\Gal(K_{C_2}/\Q)\simeq \cyc_3\rtimes \dih_4\qquad(\text{GAP id}: \langle 24,8\rangle).
\]
As noted above, we conjecture that the identity component is
\[
\ST^0(C_2) = \Unitary(1)_3,
\]
up to conjugacy in $\USp(6)$.

\begin{remark}
While we are able to give a general description of the Sato-Tate group in both cases just by looking at the $a_1$-distribution of the curves $C_i$, it should be noted that our characterization of the Sato-Tate group in terms of its identity component and the isomorphism type of its component group is far from sufficient to determine the Sato-Tate distribution.
For this we need an explicit description of the Sato-Tate group as a subgroup (up to conjugacy) of $\USp(6)$; this is addressed in the next section.
\end{remark}

\section{Determining Sato-Tate groups}\label{section: det ST}

In this section we compute the Sato-Tate groups of the curves $C_1\colon y^2=x^8+c$ and $C_2\colon y^2=x^7-cx$ for \emph{generic} values of $c\in \Q^*$. The meaning of generic will be specified in each case, but it ensures that the order of the group of components of the Sato-Tate group is as large as possible.
The Sato-Tate groups for the non-generic cases can then be obtained as subgroups.

The description of the Sato-Tate group in terms of the \emph{twisted Lefschetz group} introduced by Banaszak and Kedlaya \cite{BK15, BK16} is a useful tool for explicitly determining Sato-Tate groups (see \cite{FGL16}, for example, where this is exploited), but here we take a different approach that is better suited to our special situation.
Our strategy is to identify an elliptic quotient of each of the curves~$C_1$ and~$C_2$ and then use the classification results of \cite{FKRS12} to identify the Sato-Tate group of the complement abelian surface.
We then reconstruct the Sato-Tate group of the curves~$C_1$ and ~$C_2$ from this data.

To determine the splitting of the Jacobians of $C_1$ and $C_2$ we benefit from the fact that these are curves with large automorphism groups.
For generic $c$, the automorphism group of $C_1$ over $K_{C_1}$ has order 32 (GAP id $\langle 32,9\rangle$), and the automorphism group of $C_2$ over $K_{C_2}$ has order 24 (GAP id $\langle 24,5\rangle$).

We start by fixing the following matrix notations:
$$
I:=\begin{bmatrix}
1 & 0 \\
0 & 1
\end{bmatrix},\,
J:=\begin{bmatrix}
0 & 1 \\
-1 & 0
\end{bmatrix},\,
K:=\begin{bmatrix}
0 & i \\
i & 0
\end{bmatrix},\,
Z_{n}:=\begin{bmatrix}
e^{2\pi i/n} & 0 \\
0 & e^{-2\pi i/n}
\end{bmatrix}.
$$
Also, for $u\in \Unitary(1)$, recall the notation $U(u)$ introduced in (\ref{equation: defuni}). Whenever we consider matrices of the unitary symplectic group $\USp(6)$, we do it with respect to the symplectic form given by the matrix
\begin{equation}\label{equation: symplecticform}
H:=\begin{bmatrix}
J & 0 & 0\\
0 & J & 0\\
0 & 0 & J
\end{bmatrix}\,.
\end{equation}
If $A$ and $A'$ are two abelian varieties defined over $k$, we write $A\sim A'$ to indicate that $A$ and $A'$ are related by an isogeny defined over $k$.
Finally, we let $\zeta_3$ denote a primitive third root of unity in $\overline\Q$. 

\subsection{Sato-Tate group of $C_1\colon y^2=x^8+c$}

\begin{lemma}\label{lemma: splitC1}
Let $c\in \Q^*$ and $C_1\colon y^2=x^8+c$. Then
$$
\Jac(C_1)\sim E \times \Jac(C)\,,
$$
where $E\colon y^2=x^4+c$ and $C\colon y^2=x^5+cx$ over $\Q$. Thus $K_{C_1}=\Q(i,\sqrt{-2},c^{1/4})$.
\end{lemma}

\begin{proof}
First note that we can write nonconstant morphisms defined over~$\Q$:
\begin{equation}\label{equation: morphisms}
\begin{array}{l}
\phi_E:C_1\rightarrow E\,,\qquad \phi_E(x,y)=(x^2,y)\,,\\[6pt]
\phi_C:C_1\rightarrow C\,,\qquad \phi_C(x,y)=(x^2,xy)\,.
\end{array}
\end{equation}
We clearly have that $K_E=\Q(i)$. To see that $K_C=\Q(i,\sqrt{-2},c^{1/4})$, first set $F=\Q(c^{1/4})$, and consider the automorphism
$$
\alpha\colon C_F\rightarrow C_F, \qquad \alpha(x,y)=\left(\frac{c^{1/2}}{x},\frac{c^{3/4}}{x^3}\,y\right).
$$
Since $\alpha$ has order $2$ and is nonhyperelliptic, $C_F/\langle\alpha\rangle$ is an elliptic curve~$E'$ defined over~$F$. Poincar\'e's decomposition theorem implies that $\Jac(C)_{F}\sim E'\times E''$, where~$E''$ is an elliptic curve defined over $F$.  
Observe that we also have the automorphism
$$
\gamma\colon C_{F(i)}\rightarrow C_{F(i)}, \qquad \gamma(x,y)=(-x,i y)\,.
$$
Since $\alpha$ and $\gamma$ do not commute, we deduce that $\End(\Jac(C)_{F(i)})$ is nonabelian. It follows that $E'_{F(i)}$ and $E''_{F(i)}$ are $F(i)$-isogenous and that $\Jac(C)_{F(i)}\sim E'^2_{F(i)}$.
One may readily find an equation for the quotient curve $E'=C_F/\langle\alpha\rangle$, and, by computing its $j$-invariant, determine that $E'$ has complex multiplication by $\Q(\sqrt{-2})$.
From this we may conclude that $K_C=F(i,\sqrt {-2})$.
The asserted splitting of the Jacobian $\Jac(C_1)$ follows from the existence of the morphisms of equation (\ref{equation: morphisms}) and the fact that $E$ and $E'$ are not $\Qbar$-isogenous. This latter fact also implies that $K_{C_1}$ is the compositum of $K_E$ and $K_C$.
\end{proof}

\begin{definition} In this subsection, we say that $c\in \Q^*$ is generic if $[K_{C_1}:\Q]$ is maximal, that is, $[K_{C_1}:\Q]=16$.  Equivalently, $c\not\in \Q(i,\sqrt{-2})^{*2}$.
\end{definition}

\begin{corollary}\label{cor:STC1}
For generic $c\in \Q^*$, the Sato-Tate group of $\Jac(C_1)$ is
\begin{small}
\[
\left\langle
\begin{bmatrix}
J & 0 & 0\\
0 & J & 0\\
0 & 0 & J\\
\end{bmatrix}
,\,
\begin{bmatrix}
0 & J & 0\\
-J & 0 & 0\\
0 & 0 & I\\
\end{bmatrix}
,\,
\begin{bmatrix}
Z_8 & 0 & 0\\
0 & \overline Z_8 & 0\\
0 & 0& I\\
\end{bmatrix}
,\,
\begin{bmatrix}
U(u) & 0& 0\\
0 & U(u)& 0\\
0 & 0 & U(v)\\
\end{bmatrix}
:\, u,v\in \Unitary(1)
\right\rangle\,.
\]
\end{small}
\end{corollary}

\begin{proof} Recall the notations of Lemma \ref{lemma: splitC1}. It follows from the description of $\Jac(C)$ given in the proof of the  lemma and the results of \cite{FKRS12} that $\ST(C)$ can be presented as 
\[
\left\langle
 R:=\begin{bmatrix}
J & 0 \\
0 & J \\
\end{bmatrix}
,\,
 S:=\begin{bmatrix}
0 & J \\
-J & 0 \\
\end{bmatrix}
,\,
 T:=\begin{bmatrix}
Z_8 & 0 \\
0 & \overline Z_8 \\
\end{bmatrix}
,\,
\begin{bmatrix}
U(u) & 0\\
0 & U(u)\\
\end{bmatrix}
:\, u\in \Unitary(1)
\right\rangle\,.
\]
This is the group named $J(D_4)$ in \cite{FKRS12}. 
Since $E/\Q$ has CM, we also have
\[
\ST(E)= 
\left\langle
J,\,U(u)
: u\in \Unitary(1)
\right\rangle\,.
\]
Since $E$ is not a $\Qbar$-isogeny factor of $\Jac(C)$, we have $\ST^0(C_1)\simeq \ST^0(E)\oplus \ST^0(C)$, which proves the part of the corollary concerning the identity component. 
By Proposition \ref{proposition: 217}, we have isomorphisms
\begin{equation}\label{equation: isoGals}
\begin{array}{l}
\psi_E\colon \ST(E)/\ST^0(E)\overset{\sim}{\longrightarrow}\Gal(K_E/\Q)\,,
\\[6pt]
\psi_C\colon \ST(C)/\ST^0(C)\overset{\sim}{\longrightarrow}\Gal(K_{C}/\Q)\,,
\\[6pt]
\psi_{C_1}\colon\ST(C_1)/\ST^0(C_1)\overset{\sim}{\longrightarrow} \Gal(K_{C_1}/\Q)\,.
\end{array}
\end{equation}
The isomorphism~$\psi_E$ identifies~$J$ with the nontrivial automorphism of $K_E$, whereas the isomorphism $\psi_{C}$ identifies the images of the generators $g=R,S, T$ in
\[
\ST(C)/\ST^0(C)\simeq \langle R, S,T\rangle/\langle -1\rangle
\]
with automorphisms $\sigma=r,s,t\in\Gal(K_{C_1}/\Q)=\Gal(K_{C}/\Q)$ as indicated below: 
\begin{center}
\begin{tabular}{lllll}
$g$ & $\sigma =\psi_C(g)$ & $\sigma(i)$ & $\sigma(\sqrt{-2})$ & $\sigma(c^{1/4})$\\\midrule
$R$ & $r$ & $-i$ & $\sqrt{-2}$ & $c^{1/4}$\\
$S$ & $s$ & $i$ & $-\sqrt{-2}$ & $c^{1/4}$\\
$T$ & $t$ & $i$ & $\sqrt{-2}$ & $ic^{1/4}$\\\bottomrule
\end{tabular}
\end{center}
\medskip

Let $\mathpzc R,\mathpzc S,\mathpzc T$ be the three first generators of $\ST(C_1)$. To check the part of the theorem concerning the group of components of $\ST(C_1)$, one only needs to verify that $\mathpzc R,\mathpzc S,\mathpzc T$ generate a group of components isomorphic to
$$
\Gal(K_{C_1}/\Q) \simeq \langle \mathcal R,\,\mathcal S,\,\mathcal T\,|\, \mathcal R^2,\, \mathcal S^2,\,\mathcal T^4,\, \mathcal {RSRS},\, \mathcal{RTRT},\, \mathcal{STS}\mathcal T^3\, \rangle,
$$
and that their natural projections onto $\ST(E)/\ST^0(E)$ and onto $\ST(C)/\ST^0(C)$ are compatible with the isomorphisms of (\ref{equation: isoGals}). In this case, this amounts to noting that  $\mathpzc R,\mathpzc S,\mathpzc T$ project onto $R,S,T$ in $\ST(C)$; that the automorphism~$r$ restricts to the non-trivial element of $\Gal(K_E/\Q)$, while $\mathpzc R$ projects down to $J$ in $\ST(E)$; and that the restrictions of $s$ and $t$ to $K_E$ are trivial, as are the projections of $\mathpzc S$ and $\mathpzc T$ to $\ST(E)$.
\end{proof}

\begin{remark}
We note that even though $\Jac(C_1)\sim E\times\Jac(C)$, in the generic case the Sato-Tate group $\ST(C_1)$ is \emph{not} isomorphic to the direct sum of $\ST(E)$ and $\ST(C)$, because $\Gal(K_{C_1}/\Q)$ is not isomorphic to the direct product of $\Gal(K_E/\Q)$ and $\Gal(K_C/\Q)$.
This highlights the importance of being able to write down an explicit description for $\ST(C_1)$ in terms of generators.
\end{remark}

\begin{remark}
To treat non-generic values of $c$, one replaces $Z_8$ in the third generator for $\ST(C_1)$ in Corollary~\ref{cor:STC1} with $Z_4$ or $Z_2$ when $c\in\Q(i,\sqrt{2})^{*2}\setminus\Q(i,\sqrt{2})^{*4}$ or $c\in\Q(i,\sqrt{2})^{*4}$, respectively (in the latter case one can simply remove $\mathpzc T$ since it is already realized by $u=-1$ and $v=1$).
\end{remark}

Using the explicit representation of $\ST(C_1)$ given in Corollary~\ref{cor:STC1} one may compute moment sequences using the techniques described in \S 3.2 of \cite{FKS16}. 
The table below lists moments not only for $a_1$, but also for $a_2$ and $a_3$, where $a_i$ denotes the coefficient of $T^i$ in the characteristic polynomial of a random element of $\ST(C_1)$ distributed according to the Haar measure (these correspond to normalized $L$-polynomial coefficients of $\Jac(C_1)$):
\medskip

\begin{center}
\begin{tabular}{lrrrrrrrr}
& $M_1$ & $M_2$ & $M_3$ & $M_4$ & $M_5$ & $M_6$ & $M_7$ & $M_8$\\\midrule
$a_1\colon$ & 0 & 2 & 0 & 24 & 0 & 470 & 0 & 11235\\
$a_2\colon$ & 2 & 9 & 56 & 492 & 5172 & 59691 & 726945 & 9178434\\
$a_3\colon$ & 0 & 9 & 0 & 1245 & 0 & 284880 & 0 & 79208745\\\midrule
\end{tabular}
\end{center}
\medskip

The $a_1$ moments closely match the corresponding moment statistics listed in Table~\ref{table:C1moments} in the cases where $c$ is generic, as expected.
For a further comparison, we computed moment statistics for $a_1, a_2$ $a_3$ by applying the algorithm of \cite{HS2} to the curve $y^2=x^8+3$ over primes $p\le 2^{30}$.
The $a_1$-moment statistics listed below have less resolution than those in Table~\ref{table:C1moments}, which covers $p\le 2^{40}$ (with this higher bound we get $M_8\approx 11234$, an even better match to the value $11235$ predicted by the Sato-Tate group $\ST(C_1)$).
\medskip

\begin{center}
\begin{tabular}{lrrrrrrrr}
& $M_1$ & $M_2$ & $M_3$ & $M_4$ & $M_5$ & $M_6$ & $M_7$ & $M_8$\\\midrule
$a_1\colon$ & 0.00 & 2.00 & 0.00  & 23.98   &    0.04 &    469.26 &      1 &   11210\\
$a_2\colon$ & 2.00 & 9.00 & 55.95 & 491.22  & 5160.77 &  59527.55 & 724556 & 9143413\\
$a_3\colon$ & 0.00 & 8.99 & 0.04  & 1242.59 &   10.30 & 283980.23 &   2972 & 78866094\\\midrule
\end{tabular}
\end{center}
\medskip

\subsection{Sato-Tate group of $C_2\colon y^2=x^7-cx$}\label{section: STC2}

\begin{lemma}\label{lemma: splitC2}
Let $c\in \Q^*$ and $C_2\colon y^2=x^7-cx$. Set $F:=\Q(\zeta_3,c^{1/3})$. Then
$$
\Jac(C_2)\sim E\times A\,,
$$
where $E\colon y^2=x^3-cx$ and $A$ is an abelian surface defined over $\Q$ for which $A_F\sim E'\times E''$, where $E'$ and $E''$ are elliptic curves defined over $F$ by the equations
$$
E'\colon y^2=x^3+3c^{1/3}x\,,\qquad E''\colon y^2=x^3+3\zeta_3c^{1/3}x\,.
$$
Thus $K_{C_2}=\Q(i,c^{1/3},\sqrt{c\sqrt{-3}})$.
\end{lemma}

\begin{proof}
We can write nonconstant morphisms:
$$
\begin{array}{l}
\phi_E\colon (C_2)_F\rightarrow E_F\,,\qquad\qquad\hspace{3pt} \phi_E(x,y)=(x^3,xy)\,,\\[6pt]
\phi_{E'}\colon (C_2)_F\rightarrow E'\,,\qquad\qquad\hspace{4pt} \phi_{E'}(x,y)=\left(\frac{x^2-c^{1/3}}{x},\frac{y}{x^2}\right),\\[6pt]
\phi_{E''}\colon (C_2)_F\rightarrow E''\,,\qquad\qquad \phi_{E''}(x,y)=\left(\frac{x^2-\zeta_3c^{1/3}}{x},\frac{y}{x^2}\right).
\end{array}
$$
Note that the morphisms $\phi_E$, $\phi_{E'}$, and $\phi_{E''}$ are quotient maps given by automorphisms $\alpha_E$, $\alpha_{E'}$, and $\alpha_{E''}$ of $(C_2)_F$: 
$$
\alpha_E(x,y):=(\zeta_3x,\zeta_3^2y)\,,\quad\alpha_{E'}(x,y):=\left(\frac{-c^{1/3}}{x}, \frac{c^{2/3}y}{x^4}\right)\,,\quad\alpha_{E''}:=\alpha_E\circ\alpha_{E'}\,.
$$
To see that $\Jac(C_2)_F\sim E_F\times E'\times E''$, it is enough to check that we have an isomorphism of $F$-vector spaces of regular differential forms
$$
\Omega_{(C_2)_F}=\phi_E^*(\Omega_{E_F})\oplus \phi_{E'}^*(\Omega_{E'})\oplus \phi_{E''}^*(\Omega_{E''})\,,
$$
But this follows from the fact that $\omega_1=dx/y$, $\omega_2=x\cdot dx/y$, and $\omega_3=x^2\cdot dx/y$ constitute a basis for $\Omega_{(C_2)_F}$, together with the easy computation
$$
\phi_E^*\left(\frac{dx}{y}\right)=3\omega_ 2,\quad\phi_{E'}^*\left(\frac{dx}{y}\right)=c^{1/3}\omega_1+\omega_3,
\quad \phi_{E''}^*\left(\frac{dx}{y}\right)= \zeta_3c^{1/3}\omega_1+\omega_3.
$$
Since $E$ is defined over $\Q$, there exists an abelian surface $A$ defined over $\Q$ such that $A_F\sim E'\times E''$. To see that $K_{C_2}=\Q(i,c^{1/3},\sqrt{c\sqrt{-3}})$, first note that $F(i)\subseteq K_{C_2}$ and that $E' \sim E''$. Therefore, $K_{C_2}$ is the minimal extension of $F(i)$ over which $E$ and $E'$ become isomorphic.
Now observe that we have an isomorphism
\begin{equation}\label{equation: isoEE'}
 \psi\colon E_{F(i,\sqrt{c\sqrt{-3}})}\rightarrow E'_{F(i,\sqrt{c\sqrt{-3}})}\,,\qquad\psi(x,y)=\left(\frac{\sqrt{-3}}{c^{1/3}}x,\frac{\sqrt{-3}\sqrt{c\sqrt{-3}}}{c}y \right), 
\end{equation}
from which we see that $K_{C_2}$ is the extension of $F(i)$ obtained by adjoining the element $\sqrt{c\sqrt{-3}}$ to $F(i)$ (note: one needs to write formula in (\ref{equation: isoEE'}) carefully, otherwise one may be tempted to make $K_{C_2}$ too large).
\end{proof}

\begin{definition} In this subsection, we say that $c\in \Q^*$ is generic if $[K_{C_2}:\Q]$ is maximal, that is, $[K_{C_2}:\Q]=24$.  Equivalently, $c$ is not a cube in $\Q^*$.
\end{definition}

\begin{corollary}\label{cor:STC2}
For generic $c\in \Q^*$, the Sato-Tate group of $\Jac(C_2)$ is
\[
\left\langle
\begin{bmatrix}
J & 0 & 0\\
0 & J & 0\\
0 & 0 & J\\
\end{bmatrix}
,\,
\begin{bmatrix}
0 & K & 0\\
K & 0 & 0\\
0 & 0 & J\\
\end{bmatrix}
,\,
\begin{bmatrix}
Z_{3} & 0 & 0\\
0 & \overline Z_{3} & 0\\
0 & 0 & I\\
\end{bmatrix}
,\,
\begin{bmatrix}
U(u) & 0& 0\\
0 & U(u)& 0\\
0 & 0 & U(u)\\
\end{bmatrix}
:\, u\in \Unitary(1)
\right\rangle\,.
\]
\end{corollary}

\begin{proof}
We assume the notations of Lemma~\ref{lemma: splitC2}. Since $E$, $E'$, and $E''$ are $K_{C_2}$-isogenous, we have $\ST^0(C_2)\simeq \Unitary(1)$.
Note that $K_A=\Q(i,\zeta_3,c^{1/3})$.
We claim that $\ST(A)$ is the group named $D_{6,1}$ in \cite{FKRS12}.
As may be seen in \cite[Table 8]{FKRS12}, there are three Sato-Tate groups with identity component $\Unitary(1)$ and group of components isomorphic to $\Gal(K_A/\Q)\simeq D_6$, namely, $J(D_3)$, $D_{6,1}$, and $D_{6,2}$. We can rule out the latter option, since by \cite[Table~2]{FKRS12} this would imply that $\Gal(K_A/\Q(i))$ is a cyclic group of order~$6$, which is false.
To rule out $J(D_3)$, we need to argue along the lines of \cite[\S 4.6]{FKRS12}: Let~$F$ denote $\Q(\zeta_3,c^{1/3})$ as in Lemma \ref{lemma: splitC2}; if $\ST(A)=J(D_3)$, then $\ST(A_F)=J(C_1)$, whereas if $\ST(A)=D_{6,1}$, then $\ST(A_F)=C_{2,1}$.
By the dictionary between Sato-Tate groups and Galois endomorphism types in dimension $2$ given by Theorem~\ref{theorem: dictionary} (see \cite[Table~8]{FKRS12}), the first option would imply that $\End(A_F)\otimes_\Z \R $ is isomorphic to the Hamilton quaternion algebra $\mathbb H$, whereas the second option would yield $\End(A_F)\otimes_\Z \R \simeq \mathrm M_2(\R)$.
Since Lemma \ref{lemma: splitC2}, ensures that we are in the latter case, we must have $\ST(A)=D_{6,1}$.

For convenience, we take the following presentation of $D_{6,1}$, which is conjugate to the one given in \cite{FKRS12}: 
$$
\left\langle
 R:=\begin{bmatrix}
J & 0 \\
0 & J \\
\end{bmatrix}
,\,
S:=\begin{bmatrix}
0 & K \\
K & 0 \\
\end{bmatrix}
,\,
T:=\begin{bmatrix}
Z_{3} & 0 \\
0 & \overline Z_{3} \\
\end{bmatrix}
,\,
\begin{bmatrix}
U(u) & 0\\
0 & U(u)\\
\end{bmatrix}
: u\in \Unitary(1)
\right\rangle\,.
$$
Since $E/\Q$ has CM, we have
$$
\ST(E)= 
\left\langle
J,\,U(u)
: u\in \Unitary(1)
\right\rangle\,.
$$
By Proposition \ref{proposition: 217}, we have isomorphisms
\[
\begin{array}{l}
\psi_E\colon \ST(E)/\ST^0(E)\overset{\sim}{\longrightarrow}\Gal(K_E/\Q)\,,
\\[6pt]
\psi_A\colon \ST(A)/\ST^0(A)\overset{\sim}{\longrightarrow}\Gal(K_A/\Q)\,,
\\[6pt]
\psi_{C_2}\colon \ST(C_2)/\ST^0(C_2)\overset{\sim}{\longrightarrow} \Gal(K_{C_2}/\Q)\,.
\end{array}
\]
To prove the corollary it suffices to make these isomorphisms explicit and show that they are compatible with the projections from $\ST(C_2)$ to $\ST(E)$ and $\ST(A)$, and with the restriction maps from $\Gal(K_{C_2}/\Q)$ to $\Gal(K_E/\Q)$ and $\Gal(K_A/\Q)$.

The isomorphism $\psi_E$ identifies the image of $J$ in $\ST(E)/\ST^0(E)$ with the non-trivial element of $\Gal(K_E/\Q)$, while the isomorphism $\psi_A$ identifies the images of the generators $g=R,S, T$ in
\[
\ST(A)/\ST^0(A)\simeq \langle R, S,T\rangle/\langle -1\rangle
\]
with automorphisms $\sigma=r,s,t\in\Gal(K_A/\Q)$ as indicated below: 

\begin{center}
\begin{tabular}{lllll}
$g$ & $\sigma =\psi_A(g)$ & $\sigma(i)$ & $\sigma(\zeta_3)$ & $\sigma(c^{1/3})$\\\midrule
$R$ & $r$ & $-i$ & $\zeta_3^2$ & $c^{1/3}$\\
$S$ & $s$ & $-i$ & $\zeta_3$ & $c^{1/3}$\\
$T$ & $t$ & $i$ & $\zeta_3$ & $\zeta_3c^{1/3}$\\\bottomrule
\end{tabular}
\end{center}
\medskip

If we now let $\mathpzc R, \mathpzc S, \mathpzc T$ denote the first three generators of $\ST(C_2)$ listed in the corollary, $\psi_{C_2}$ identifies their images in $\ST(C_2)/\ST^0(C_2)$ with elements of $\Gal(K_A/\Q)$ as indicated below, where $\delta=\sqrt{c\sqrt{-3}}$:

\begin{center}
\begin{tabular}{llllll}
$g$ & $\sigma =\psi_{C_2}(g)$ & $\sigma(i)$ & $\sigma(\zeta_3)$ & $\sigma(c^{1/3})$ & $\sigma(\delta)$\\\midrule
$\mathpzc R$ & $\mathpzc r$ & $-i$ & $\zeta_3^2$ & $c^{1/3}$& $i\delta$\\
$\mathpzc S$ & $\mathpzc s$ & $-i$ & $\zeta_3$ & $c^{1/3}$ & $\delta$\\
$\mathpzc T$ & $\mathpzc t$ & $i$ & $\zeta_3$ & $\zeta_3\,c^{1/3}$ & $\delta$\\\bottomrule
\end{tabular}
\end{center}
\medskip

We note that, unlike their restrictions $r$ and $s$, the automorphisms $\mathpzc r$ and $\mathpzc s$ do not commute, they generate a dihedral group of order~$8$ inside $\Gal(K_{C_2}/\Q)$.
The three automorphisms $\mathpzc r, \mathpzc s, \mathpzc t$ together generate $\Gal(K_{C_2}/\Q)$.
Their restrictions to $K_A$ are the generators $r, s, t$ for $K_A$, and $R,S,T$ are the projections of $\mathpzc R, \mathpzc S,\mathpzc T$ to $\ST(A)$.
The automorphisms $\mathpzc r$ and $\mathpzc s$ both restrict to the non-trivial element of $\Gal(K_E/\Q)$, and both $\mathpzc R$ and $\mathpzc S$ project down to $J$ in $\ST(E)$.
The restriction of $\mathpzc t$ to $K_E$ is trivial, as is the projection of $\mathpzc T$ to $\ST(E)$.
To complete the proof it suffices to verify that the map
\[
\ST(C_2)/\ST^0(C_2)\simeq \langle \mathpzc R, \mathpzc S, \mathpzc T\rangle/\langle -1\rangle \xrightarrow{\quad\psi_{C_2}\quad} \langle \mathpzc r, \mathpzc s, \mathpzc t\rangle\simeq \Gal(K_{C_2}/\Q)
\]
we have explicitly defined is indeed an isomorphism.  One can check that both sides are isomorphic to the finitely presented group
\begin{align*}
\langle \mathcal{R,S,T}\,|\,\mathcal{R}^2,\mathcal{S}^2,\mathcal{T}^3,\mathcal{RSRSRSRS,RTRT,STST}^2\rangle,
\end{align*}
via maps that send generators to corresponding generators (in the order shown).
\end{proof}

\begin{remark}
To treat non-generic values of $c$, simply remove the third generator containing $Z_3$ from the list of generators for $\ST(C_2)$ in Corollary~\ref{cor:STC2} when $c$ is a cube in $\Q^*$.
\end{remark}

Using the explicit representation of $\ST(C_2)$ given in Corollary~\ref{cor:STC2}, one may compute moments sequences for the characteristic polynomial coefficients $a_1,a_2,a_3$ using the techniques described in \S 3.2 of \cite{FKS16}; the first eight moments are listed below:
\medskip

\begin{center}
\begin{tabular}{lrrrrrrrr}
  & $M_1$ & $M_2$ & $M_3$ & $M_4$ & $M_5$ & $M_6$ & $M_7$ & $M_8$\\\midrule
$a_1\colon$ & 0 & 2 & 0 & 30 & 0 & 720 & 0 & 20650\\
$a_2\colon$ & 2 & 10 & 75 & 784 & 9607 & 126378 & 1721715 & 23928108\\
$a_3\colon$ & 0 & 11 & 0 & 2181 & 0 & 660790 & 0 & 224864661\\\midrule
\end{tabular}
\end{center}
\medskip

The $a_1$ moments closely match the corresponding moment statistics listed in Table~\ref{table:C2moments} in the cases where $c$ is generic, as expected.
We also computed moment statistics for $a_1, a_2$ and $a_3$ by applying the algorithm of \cite{HS2} to the curve $y^2=x^7-2x$ over primes $p\le 2^{30}$.  The $a_1$-moment statistics listed below have less resolution than Table~\ref{table:C2moments}, which covers $p\le 2^{40}$ (with this higher bound we get $M_8\approx 20649$, very close to the value $20650$ predicted by $\ST(C_2)$).
\medskip

\begin{center}
\begin{tabular}{lrrrrrrrr}
& $M_1$ & $M_2$ & $M_3$ & $M_4$ & $M_5$ & $M_6$ & $M_7$ & $M_8$\\\midrule
$a_1\colon$ & 0.00 &  2.00 &  0.00 &   30.00 &    0.04 &    719.62 &       2 &     20636\\
$a_2\colon$ & 2.00 & 10.00 & 74.97 &  783.59 & 9600.64 & 126281.75 & 1720266 &  23906297\\
$a_3\colon$ & 0.00 & 11.00 &  0.04 & 2179.67 &   19.68 & 660247.53 &    8549 & 224645654\\\midrule
\end{tabular}
\end{center}

\section{Galois endomorphism types}\label{section: GT}

As recalled in \S\ref{section: background}, up to dimension 3, the Galois endomorphism type of an abelian variety over a number field is determined by its Sato-Tate group. In this section, we derive the Galois endomorphism type of $\Jac(C_2)$ from $\ST(C_2)$ for generic values of $c$ (in the sense of \S\ref{section: STC2}).
The case of $\Jac(C_1)$, although leading to slightly larger diagrams, is completely analogous. 

Let $G:=\ST(C_2)$ and $V:=\End(\Jac(C_2)_{K_{C_2}})$, and set $V_\C:= V\otimes_\Z\C$ and $V_\R:= V\otimes_\Z\R$. As described in the proof of \cite[Prop. 2.19]{FKRS12}:
\begin{itemize}
\item $V_\C$ is the subspace of $\mathrm M_6(\C)$ fixed by the action of $G^0$;
\item $V_\R$ is the subspace of $V_\C$, of half the dimension, over which the Rosati form is positive definite;
\item If $L/\Q$ is a subextension of $K_{C_2}/\Q$, corresponding to the subgroup $N\subseteq \Gal(K_{C_2}/\Q)\simeq G/G^0$, then $\End(\Jac(C_2)_L)\otimes_\Z\R\simeq V_\R^N$.
\end{itemize}
The matrices $\Phi\in\mathrm M_6(\C)$ commuting with $G^0\simeq \Unitary(1)$, embedded in $\USp(6)$, are matrices of the form $\Phi=(\phi_{i,j})$ with $i,j\in[1,6]$ such that $\phi_{i,j}\in \C$ is $0$ unless $i\equiv j \mod 2$. The condition of the Rosati form being positive definite on $V_\R$ amounts to requiring that 
$$
\Trace(\Phi H^t \Phi^t H )\geq 0
$$
for every $\Phi\in V_\R$, where $H$ is the symplectic matrix given in (\ref{equation: symplecticform}). Imposing the above condition on $\Phi$, we find that
$$
\Phi=\begin{pmatrix}
\alpha & 0 & \beta & 0 & \gamma & 0 \\
0 & \overline \alpha & 0& \overline \beta & 0& \overline \gamma \\
\delta & 0 & \epsilon & 0& \phi & 0 \\
0 & \overline \delta & 0& \overline \epsilon & 0& \overline \phi \\
\lambda & 0 & \mu & 0& \nu & 0 \\
0 & \overline \lambda & 0 & \overline \mu & 0& \overline \nu \\
\end{pmatrix}\qquad \text{with }\alpha ,\beta,\dots, \mu,\nu\in \C\,.
$$
We thus deduce that $V_\R\simeq \mathrm M_3(\C)$.

We now proceed to determine the sub-$\R$-algebras of $V_\R$ fixed by each of the subgroups of $\Gal(K_{C_2}/\Q)\simeq G/G^0$. 
With notations as in the proof of Corollary \ref{cor:STC2}, these subgroups are listed (up to conjugation) in Figure~\ref{Table: latticesubgroupsC2}, where normal subgroups are marked with a~$^*$. We can then reconstruct the Galois type of $\Jac(C_2)$ (see Figure \ref{Table: latticeC2}) from the information in Table \ref{Table: conditionsphi}.

\begin{figure}[h]
$$
\xymatrix{
24& & & \langle \mathpzc r,\mathpzc t,\mathpzc s \rangle^* \ar@{-}[d] \ar@{-}[ddl] \ar@{-}[dr]   \ar@{-}[dl] & & \\
12& & \langle \mathpzc t,\mathpzc r, (\mathpzc r\mathpzc s)^2 \rangle^* \ar@{-}[dddl] \ar@{-}[ddr]\ar@{-}[ddrr]& \langle \mathpzc t, \mathpzc r\mathpzc s\rangle^* \ar@{-}[dddl]\ar@{-}[ddr] & \langle \mathpzc t,\mathpzc s, (\mathpzc r\mathpzc s)^2\rangle^* \ar@{-}[dd]\ar@{-}[dddl]\ar@{-}[ddr]& \\
8 & & \langle \mathpzc r,\mathpzc s \rangle \ar@{-}[ddl] \ar@{-}[dd] \ar@{-}[d]\ar@{-}[ddr]&   & & \\
6 & & & \langle \mathpzc t,\mathpzc r\rangle \ar@{-}[ddr] \ar@{-}[dddl]& \langle \mathpzc t, (\mathpzc r\mathpzc s)^2\rangle^* \ar@{-}[dd]\ar@{-}[dddl]  & \langle \mathpzc t,\mathpzc s\rangle\ar@{-}[ddl]\ar@{-}[dddl] \\
4 & \langle (\mathpzc r\mathpzc s)^2,\mathpzc r \rangle \ar@{-}[ddr] \ar@{-}[ddrr]& \langle \mathpzc r\mathpzc s \rangle \ar@{-}[ddr]&  \langle (\mathpzc r\mathpzc s)^2,\mathpzc s\rangle^* \ar@{-}[dd]\ar@{-}[ddr]& & \\
3 & & & & \langle \mathpzc t\rangle^* \ar@{-}[ddl]  & \\ 
2& & \langle \mathpzc r \rangle \ar@{-}[dr] & \langle (\mathpzc r\mathpzc s)^2\rangle^* \ar@{-}[d] &  \langle \mathpzc s \rangle \ar@{-}[dl]& \\
1& & & \langle 1\rangle^*  & & 
}
$$
\caption{Lattice of subgroups of $\Gal(K_{C_2}/\Q)$.}\label{Table: latticesubgroupsC2}
\end{figure}
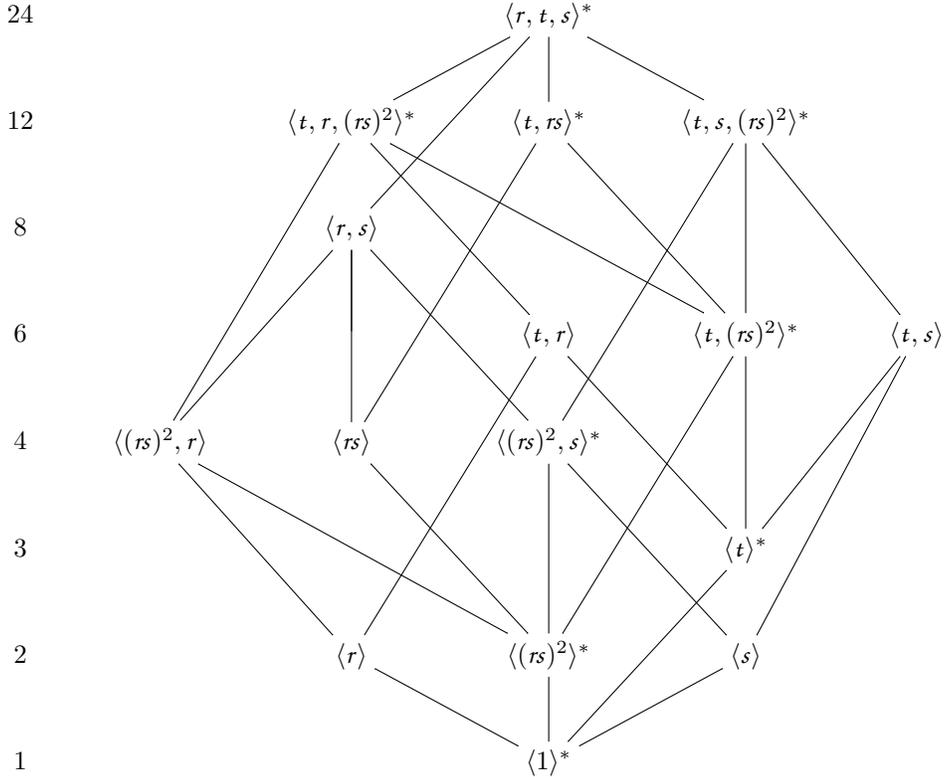

\begin{center}
\begin{table}[h]
\begin{tabular}{lll}
$N$  & Condition on $\Phi$ & $V_\R^{N}$\\\midrule
$\langle \mathpzc R\rangle $ & $\alpha,\beta,\dots,\mu,\nu\in \R$& $\mathrm M_3(\R)$\\
$\langle \mathpzc S\rangle $ & $\alpha=\overline \epsilon, \beta=\overline \delta, \phi=i\overline \gamma,  \mu=-i \overline \lambda, \nu\in \R$& $\mathrm M_3(\R)$ \\
$\langle \mathpzc T\rangle $ & $\beta=\gamma=\delta=\phi=\lambda=\mu=0$& $\C\times \C\times \C$\\
$\langle (\mathpzc {RS})^2\rangle $ & $\gamma=\phi=\lambda=\mu=0$& $\mathrm M_2(\C)\times \C$\\
$\langle (\mathpzc {RS})^2, \mathpzc S\rangle $ & $\gamma=\phi=\lambda=\mu=0,\alpha=\overline \epsilon, \beta=\overline \delta,\nu\in \R$& $\mathrm M_2(\R)\times \R$\\
$\langle \mathpzc {RS}\rangle $ & $\gamma=\phi=\lambda=\mu=0,\alpha=\epsilon,\beta=\delta$& $\C\times \C\times \C$\\\bottomrule
\end{tabular}
\bigskip
\caption{Some subgroups of $\Gal(K_{C_2}/\Q)$ with the respective fixed sub-$\R$-algebras of $V_\R$.}\label{Table: conditionsphi}
\end{table}
\end{center}

Obtaining the data in Table \ref{Table: conditionsphi} is a straight-forward exercise, let us make just a few specific  comments: 

\begin{itemize}
\setlength{\itemsep}{4pt}
\item $N=\langle \mathpzc S\rangle$: One easily checks that the matrices $\Phi$ satisfying the required condition form a simple nondivision $\R$-algebra; by Wedderburn's structure theorem, it is of the form $\mathrm M_d(D)$, for some division algebra $D$ and $d>1$; since its real dimension is $9$, we must have $d=3$ and $D=\R$. 
\item $N=\langle (\mathpzc{RS})^2,\mathpzc S\rangle$: Note that the $\R$-algebra
\[
\mathpzc H:=\left\{A_{\alpha,\beta}:=\begin{pmatrix}\alpha & \beta \\ \overline\beta &\overline\alpha \end{pmatrix} : \alpha,\beta \in \C \right\}
\]
is isomorphic to $\mathrm M_2(\R)$. Indeed, if $\alpha=\alpha_1+i\alpha_2$ and $\beta=\beta_1+i\beta_2$, then
\[
\psi\colon \mathpzc H\rightarrow \mathrm M_2(\R)\,,\qquad \psi(A_{\alpha,\beta})=
\begin{pmatrix}
\alpha_1 +\beta_1 & -\alpha_2 +\beta_2\\ \alpha_2 +\beta_2 & \alpha_1 -\beta_1
\end{pmatrix}
\]
provides the required isomorphism. Alternative, one can reach the same conclusion by noting that $\mathpzc H$ is the only non-commutative $\R$-algebra of dimension 4 with zero divisors.
\item  $N=\langle \mathpzc {RS}\rangle$: Note that the $\R$-algebra
\[\mathpzc H:=\left\{A_{\alpha,\beta}:=\begin{pmatrix}\alpha & \beta \\  \beta &  \alpha \end{pmatrix} : \alpha,\beta \in \C \right\}
\]
is isomorphic to $\C\times \C$ by means of
\[
\psi\colon \mathpzc H\rightarrow \mathrm \C\times \C\,,\qquad \psi(A_{\alpha,\beta})=
(\alpha+\beta,\alpha-\beta)\,.
\]
\end{itemize}

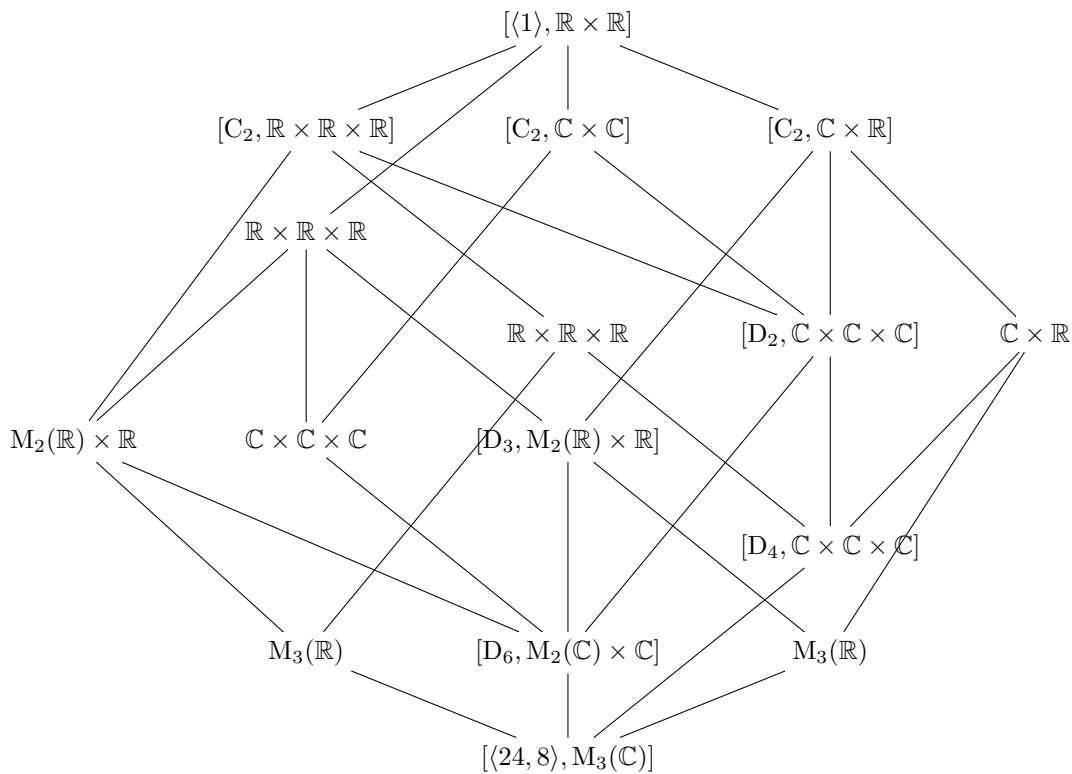
\begin{figure}[h]
\caption{Galois endomorphism types of $\Jac(C_2)$. Lattice of fixed sub-$\R$-algebras corresponding to the subgroups of Figure \ref{Table: latticesubgroupsC2}.}\label{Table: latticeC2}
$$
\xymatrix{
 & & [\langle 1\rangle, \R\times \R ] \ar@{-}[d]  \ar@{-}[dr]   \ar@{-}[dl]\ar@{-}[ddl] & & \\
 & [\cyc_2,\R\times \R\times \R] \ar@{-}[dddl]\ar@{-}[ddr] \ar@{-}[ddrr]& [\cyc_2,\C\times \C] \ar@{-}[dddl]\ar@{-}[ddr] & [\cyc_2,\C\times \R] \ar@{-}[dd]\ar@{-}[dddl]\ar@{-}[ddr]& \\
 & \R\times \R\times \R \ar@{-}[ddl]\ar@{-}[dd]\ar@{-}[ddr]&   & & \\
 &  &\R\times \R\times \R \ar@{-}[dddl]\ar@{-}[ddr]& [\dih_2,\C \times \C\times \C]\ar@{-}[dd]\ar@{-}[dddl] & \C\times \R \ar@{-}[dddl]\ar@{-}[ddl]\\
\mathrm M_2(\R)\times \R \ar@{-}[ddr]\ar@{-}[ddrr]&\C\times \C\times \C\ar@{-}[ddr]& [\dih_3,\mathrm M_2(\R)\times \R]  \ar@{-}[dd] \ar@{-}[ddr]& & \\
 & & & [\dih_4,\C\times \C\times \C] \ar@{-}[ddl]  & \\ 
 & \mathrm M_3(\R)\ar@{-}[dr] & [\dih_6,\mathrm M_2(\C)\times \C] \ar@{-}[d] & \mathrm M_3(\R)\ar@{-}[dl] &   \\
 & & [\langle 24,8\rangle, \mathrm M_3(\C) ] & & \\  }
$$
\end{figure}


\begin{thebibliography}{McK-Sta}

\bibitem[Bac90]{Bach90}
E. Bach, \href{http://www.ams.org/journals/mcom/1990-55-191/S0025-5718-1990-1023756-8/S0025-5718-1990-1023756-8.pdf}{\textit{Explicit bounds for primality testing and related problems}}, Math. Comp. \textbf{55} (1990), 335--380.

\bibitem[Bas04]{Bas04}
J.M. Basilla, \href{https://projecteuclid.org/euclid.pja/1116442240}{\textit{On the solution of $x^2+dy^2=m$}}, Proc. Japan Acad. Ser. A Math. Sci. \textbf{80} (2004), 40--41.

\bibitem[BEW98]{BEW98}
B. Berndt, R. Evans, K. Williams, \href{http://www.ams.org/mathscinet-getitem?mr=1625181}{\textit{Gauss and Jacobi sums}}, Wiley, 1998.

\bibitem[BK15]{BK15}
G. Banaszak and K.S. Kedlaya, \href{http://www.iumj.indiana.edu/IUMJ/fulltext.php?artid=5438&year=2015&volume=64}{\textit{An algebraic Sato-Tate group and Sato-Tate conjecture}}, Indiana Univ. Math. J. \textbf{64} (2015), 245--274.


\bibitem[BZ10]{BZ10}
R.P. Brent and P. Zimmerman, \href{http://link.springer.com/chapter/10.1007/978-3-642-14518-6_10}{\textit{An $O(\M(n)\log n)$ algorithm for the {J}acobi symbol}}, Algorithmic Number Theory 9th International Symposium (ANTS IX), LNCS \textbf{6197}, Springer, 2010, 83--95.

\bibitem[Coh93]{Coh93}
H. Cohen, \href{http://link.springer.com/book/10.1007/978-3-662-02945-9}{\textit{A course in computational algebraic number theory}}, Springer, 1993.

\bibitem[Erd61]{Erd61}
P. Erd\"os, \href{http://www.ams.org/mathscinet-getitem?mr=144869}{\textit{Remarks on number theory. I}}, Mat. Lapok \textbf{12} (1961) 10--17.

\bibitem[FKRS12]{FKRS12}
F. Fit\'e, K.S. Kedlaya, V. Rotger, and A.V. Sutherland, \href{http://journals.cambridge.org/abstract_S0010437X12000279}{\textit{Sato-Tate distributions and Galois endomorphism modules in genus $2$}}, Compos. Math. \textbf{148} (2012), 1390--1442.

\bibitem[FGL16]{FGL16}
F. Fit\'e, J. Gonz\'alez, J-C. Lario, \href{https://cms.math.ca/10.4153/CJM-2015-028-x}{\textit{Frobenius distribution for quotients of Fermat curves of prime exponent}}, Canad. J. Math, published online 2016-02-05, to appear in print.

\bibitem[FKS13]{FKS16}
F. Fit\'e, K.S. Kedlaya, and A.V. Sutherland, \href{http://arxiv.org/abs/1212.0256}{\textit{Sato-Tate groups of some weight $3$ motives}}, arXiv:1212.0256.

\bibitem[FKT04]{FKT04}
E. Furukawa, M. Kawazoe, and T. Takahashi, \href{http://link.springer.com/chapter/10.1007\%2F978-3-540-24654-1_3}{\textit{Counting points for hyperelliptic curves of type $y^2=x^5+ax$ over finite prime fields}}, in Selected Areas in Cryptography, LNCS \textbf{3006}, Springer, 2004, 26--41.

\bibitem[GG13]{GG13}
J. von zur Gathen and J. Gerhard, \href{http://ebooks.cambridge.org/ebook.jsf?bid=CBO9781139856065}{\textit{Modern computer algebra}}, 3rd ed., Cambridge University Press, 2013.

\bibitem[HS14a]{HS1}
D. Harvey and A.V. Sutherland, \href{http://dx.doi.org/10.1112/S1461157014000187}{\textit{Computing Hasse--Witt matrices of hyperelliptic curves in average polynomial time}}, in Algorithmic Number Theory 11th International Symposium (ANTS XI), LMS J. Comput. Math. (2014), 257-273.

\bibitem[HS14b]{HS2}
D. Harvey and A.V. Sutherland, \href{http://arxiv.org/abs/1410.5222}{\textit{Computing Hasse--Witt matrices of hyperelliptic curves in average polynomial time, II}}, arXiv:1410.5222.

\bibitem[Man61]{M61}
Yu.~I. Manin, \href{http://www.ams.org/mathscinet-getitem?mr=124324}{\textit{The {H}asse-{W}itt matrix of an algebraic curve}}, AMS Translations, Series 2 \textbf{45} (1965), 245--264, (originally published in {Izv. Akad. Nauk SSSR Ser. Mat.} \textbf{25} (1961) 153--172).

\bibitem[PW03]{PW03}
X. Wang and V. Pan, \href{http://epubs.siam.org/doi/abs/10.1137/S0097539702408636}{\textit{Acceleration of {E}uclidean algorithm and rational number reconstruction}}, SIAM J. Comput. \textbf{32} (2003), 548--556.

\bibitem[SS71]{SS71}
A.~Sch{\"o}nhage and V.~Strassen, \href{http://www.ams.org/mathscinet-getitem?mr=292344}{\textit{Schnelle {M}ultiplikation grosser {Z}ahlen}}, Computing (Arch. Elektron. Rechnen) \textbf{7} (1971), 281--292.

\bibitem[Sch85]{Sch85}
R. Schoof, \href{http://www.ams.org/journals/mcom/1985-44-170/S0025-5718-1985-0777280-6/}{\textit{Elliptic curves over finite fields and the computation of square roots mod~$p$}}, Math. Comp. \textbf{44} (1985), 483--494.

\bibitem[SS14]{SS14}
I. Shparlinski and A.V. Sutherland, \href{http://journals.cambridge.org/action/displayAbstract?fromPage=online&aid=9682813&fulltextType=RA&fileId=S1461157015000017}{\textit{On the distribution of Atkin and Elkies primes for reductions of elliptic curves on average}}, LMS J. Comput. Math. \textbf{15} (2015), 308--322.

\bibitem[Ser12]{Ser12}
J.-P. Serre, \href{https://www.crcpress.com/Lectures-on-NXp/Serre/9781466501928}{\textit{Lectures on $N_X(p)$}}, A.K. Peters/CRC Press, 2012.

\bibitem[Sil94]{Si94}
J. Silverman, \href{http://link.springer.com/book/10.1007\%2F978-1-4612-0851-8}{\textit{Advanced topics in the arithmetic of elliptic curves}}, Springer, 1994.

\bibitem[Sut11]{Sut11}
A.V. Sutherland, \href{http://www.ams.org/journals/mcom/2011-80-273/S0025-5718-10-02356-2/}{\textit{Structure computation and discrete logarithms in finite abelian $p$-groups}}, Math. Comp. \textbf{80} (2011), 477-500.

\bibitem[Yui78]{Y78}
Noriko Yui, \href{http://www.sciencedirect.com/science/article/pii/0021869378902478}{\textit{On the {J}acobian varieties of hyperelliptic curves over fields of characteristic {$p>2$}}}, J. Algebra \textbf{52} (1978), 378--410.
  
\end{thebibliography}
\end{document}